\documentclass[10pt]{article}

\usepackage{amsmath,amsfonts,amssymb,amsthm}
\usepackage{thmtools,verbatim,url}
\usepackage{mathtools,tikz-cd,tikz}
\usepackage{enumerate,enumitem} 
\usepackage{lipsum}

\usepackage{mathrsfs}
\usepackage{graphicx}
\usepackage{blkarray}
\usepackage[unicode=true,colorlinks=true,linkcolor=blue,urlcolor=blue, citecolor=blue]{hyperref}

\usepackage{calrsfs}
\DeclareMathAlphabet{\pazocal}{OMS}{zplm}{m}{n}

\usepackage{ifpdf} 
\usepackage[all]{xy}
\ifpdf
\else
\fi

\usepackage{cleveref}

\newtheorem{theorem}{Theorem}[section]
\newtheorem{lemma}[theorem]{Lemma}
\newtheorem{proposition}[theorem]{Proposition}
\newtheorem{corollary}[theorem]{Corollary}

\theoremstyle{definition}
\newtheorem{example}[theorem]{Example}
\newtheorem{question}[theorem]{Question}

\theoremstyle{remark}
\newtheorem{remark}[theorem]{Remark}
\newtheorem{construction}[theorem]{Construction}

\newtheorem{definition}[theorem]{Definition}  

\numberwithin{equation}{section}
\numberwithin{table}{section}
\numberwithin{figure}{section}

\newcommand{\GG}{{\mathbb G}}	
\newcommand{\PP}{\mathbb{P}}         
\newcommand{\QQ} {{\mathbb Q}}		
\newcommand{\RR} {{\mathbb R}}		
\newcommand{\ZZ} {{\mathbb Z}}	
\newcommand{\GGm}{{\mathbb G}_m}	
\newcommand{\FF}{\ensuremath{\mathbb{F}}}
\renewcommand{\AA}{\ensuremath{\mathbb{A}}} 
\newcommand{\KK}{\ensuremath{\mathbb{K}}}

\DeclareMathOperator{\Gr}{Gr}

\DeclareMathOperator{\Aut}{Aut}
 
\DeclareMathOperator{\Span}{span}
\DeclareMathOperator{\id}{id}

\DeclareMathOperator{\Spec}{{Spec}}
\DeclareMathOperator{\Proj}{{Proj}}
\DeclareMathOperator{\inw}{in}
\DeclareMathOperator{\Trop}{Trop}

\DeclareMathOperator{\conv}{conv}

\DeclareMathOperator{\Grob}{G}

\DeclareMathOperator{\St}{Star}

\DeclareMathOperator{\TGr}{TGr}
\DeclareMathOperator{\Dr}{Dr}
\DeclareMathOperator{\PGL}{PGL}
\DeclareMathOperator{\sch}{sch}
\DeclareMathOperator{\Pic}{Pic}

\DeclareMathOperator{\face}{face}
\DeclareMathOperator{\Pa}{Pa}
\DeclareMathOperator{\sat}{sat}
\DeclareMathOperator{\sgn}{sgn}

\DeclareMathOperator{\Sec}{S}
\DeclareMathOperator{\Sym}{Sym}

\newcommand{\GrU}[2]{\Gr_0(#1,#2)}
\newcommand{\GrC}[2]{\Gr(#1,#2)}
\newcommand{\TGrU}[2]{\TGr_0(#1,#2)}
\newcommand{\DrU}[2]{\Dr(#1,#2)}
\newcommand{\XU}[2]{X_0(#1,#2)}
\newcommand{\XC}[2]{X(#1,#2)}
\newcommand{\XCL}[2]{X_{\Sec}(#1,#2)}
\newcommand{\SF}[1]{\Sigma_{\Sec}(#1)}
\newcommand{\SFU}[2]{\Sigma_{\Sec}(#1,#2)}

\newcommand{\LFU}[2]{\mathcal{S}_{#1,#2}}
\newcommand{\mat}[2]{(#1,#2)}
\newcommand{\ETrop}{\mathfrak{Trop}}
\newcommand{\charF}{\operatorname{char}}
\newcommand{\supp}{\operatorname{supp}}
\newcommand{\TC}[1]{\operatorname{TC}_{#1}}

\newcommand{\kk}{\mathbf{k}}

\newcommand{\calP}{\pazocal}

\def\cA{{\calP A}}
\def\cB{{\calP B}}
\def\cC{{\calP C}}

\def\cF{{\calP F}}
\def\cG{{\calP G}}

\def\cL{{\calP L}}

\def\cR{{\calP R}}

\def\cT{{\calP T}}

\newcommand \sign {\operatorname{sign}}

\newcommand{\init}{\operatorname{in}}

\newcommand{\qp}[3]{P_{#1}(#2,#3)}

\usepackage{tikz}
\usepackage{mathtools,tikz-cd}  

\tikzset{%
	add/.style args={#1 and #2}{to path={%
			($(\tikztostart)!-#1!(\tikztotarget)$)--($(\tikztotarget)!-#2!(\tikztostart)$)%
			\tikztonodes}}
}

\newcommand{\tsc}[1]{\ensuremath{\Gr_{#1}}}

\newcommand{\subd}[2]{\ensuremath{\Delta_{#1,#2}}}
\newcommand{\grdiv}[2]{\ensuremath{\Gamma_{#1,#2}}}
\newcommand{\grdivF}[3]{\ensuremath{\Gamma_{#1,#2}^{#3}}}

\newcommand{\invLim}[2]{\ensuremath{\lim_{\substack{\longleftarrow\\#1 \in #2}}}}
\newcommand{\dr}[1]{\ensuremath{\operatorname{Dr}_{#1}}}

\newcommand{\MT}{\texttt{Macaulay2}}
\newcommand{\pmk}{\texttt{polymake}}
\newcommand{\sage}{\texttt{sage}}
\newcommand{\gfan}{\texttt{gfan}}

\newcommand{\zspan}[1]{\ensuremath{\ZZ\!\cdot\!#1}}
\newcommand{\rspan}[1]{\ensuremath{\RR\!\cdot\!#1}}

\newcommand{\relint}{\operatorname{rel\,int}}

\newcommand{\kSch}{\mathop{\mbox{$\kk$-$\sch$}}}

\title{Initial degenerations of Grassmannians}

\author{Daniel Corey \footnote{UW-Madison, Van Vleck Hall, 480 Lincoln Dr. Madison, WI 53706. dcorey@math.wisc.edu }}

\begin{document}
\maketitle

\begin{abstract}
	We construct closed immersions from initial degenerations of $\Gr_{0}(d,n)$---the open cell in the Grassmannian $\GrC{d}{n}$ given  by the nonvanishing of all Pl\"ucker coordinates---to limits of thin Schubert cells associated to diagrams induced by the face poset of the corresponding tropical linear space. These are isomorphisms when $(d,n)$ equals $(2,n)$, $(3,6)$ and $(3,7)$. As an application we prove $\Gr_0(3,7)$  is sch\"on, and the Chow quotient of $\Gr(3,7)$ by the maximal torus in $ \operatorname*{PGL}(7)$ is the log canonical compactification of the moduli space of 7 points in $\mathbb{P}^2$ in linear general position, making progress on a conjecture of Hacking, Keel, and Tevelev.

\smallskip
	
	
	\smallskip
	
\end{abstract}

\maketitle

\section{Introduction}

Let $\GrC{d}{n}$ be the Grassmannian of $d$-dimensional subspaces of $\kk^n$, for an algebraically closed field $\kk$, and  $\GrU{d}{n}$  the open cell  given by the nonvanishing of all Pl\"ucker coordinates. We consider its tropicalization $\TGrU{d}{n}$ through two frameworks. Via Gr\"obner theory, $\TGrU{d}{n}$ consists of those $w\in \wedge^d\RR^n$ such that the initial degeneration $\inw_w\GrU{d}{n}$ is nonempty.  Alternatively, $\TGrU{d}{n}$ has a modular interpretation
as the space of $d$-dimensional tropical linear subspaces of $\RR^n$ that are realizable over valued extensions of $\kk$ \cite{SpeyerSturmfels2004a}.  The goal of this paper is to study initial degenerations of $\GrU{d}{n}$  via their relation to tropical linear spaces. 

Suppose $w\in \TGrU{d}{n}$ and $L_w$ is the corresponding tropical  linear space. Then $w$ induces a regular subdivision $\Delta_w$ of the hypersimplex $\Delta(d,n)\subset \RR^n$ into matroid polytopes, and there is a bijection between the bounded cells of $L_w$ and the internal cells of $\Delta_w$, reversing the face order \cite{Speyer2008}.
Equipping $\TGrU{d}{n}$ with its Gr\"obner fan structure,  $\inw_w\GrU{d}{n}$ and $\Delta_w$ depend only on the cone of $\TGrU{d}{n}$ that contains $w$ in its relative interior  \cite{Tevelev2007}. 

The collection of all subspaces realizing a matroid $M$ defines a locally closed subscheme  $\Gr_{M}\subset \GrC{d}{n}$ called the \textit{thin Schubert cell} of $M$.
Let us describe an inverse system of thin Schubert cells associated to the matroid subdivision $\Delta_w$. 
Given  a matroid polytope $Q\subset \Delta(d,n)$, write $M_Q$ for its matroid and  $\rho_{M_{Q}}$ for the rank function.   Any facet of $Q$ has  the form $Q' = Q\cap \{\sum_{i\in\eta} x_i = \rho_{M}(\eta)\}$ for some $\eta\subset [n]:=\{0,1,\ldots,n-1\}$, and $M_{Q'}$ decomposes as a direct sum of the contraction $M_{Q}/\eta$ and restriction $M_{Q}|\eta$ \cite{GelfandSerganova}. 
If $F\in \GrC{d}{n}$ realizes $M_Q$ and $\mu=[n]\setminus \eta$, then $F\cap \kk^{\mu}$ and $F/(F\cap \kk^{\mu})$  realize $M_{Q}/\eta$ and $M_{Q}|\eta$, respectively.  We have a morphism
\begin{align*}
\Gr_{M_Q} \to \Gr_{M_{Q'}} \hspace{30pt}
F \mapsto (F\cap \kk^{\mu}) \oplus F/(F\cap \kk^{\mu}).
\end{align*}
Thus $\{\Gr_{M_Q} \, | \, Q\in \Delta_{w}\}$ defines an inverse system, and we may form the limit $\varprojlim_{Q\in\Delta_w}\Gr_{M_Q}$. 
\begin{theorem}
	\label{thm:closedEmbeddingIntro}
	For $w\in \TGrU{d}{n}$, there is a closed immersion
	\begin{equation*}
	\psi_w:\inw_w\GrU{d}{n}\hookrightarrow \varprojlim_{Q\in\Delta_w} \Gr_{M_Q}.
	\end{equation*} 
\end{theorem}
\noindent The morphisms $\Gr_{M_{Q}} \to \Gr_{M_{Q'}}$ and limit ~$\varprojlim_{Q\in\Delta_w}\Gr_{M_Q}$ originally appear in \cite{Lafforgue2003}. This limit parameterizes collections of subspaces $\{F_Q \in \Gr_{M_Q}\, | \, Q\in \Delta_w\}$, such that, if $Q_1$ and $Q_2$ share a common face defined by $\sum_{i\in \eta}x_{i} = \rho_{M_{Q_1}}(\eta) = d-\rho_{M_{Q_2}}(\mu)$ with $\mu=[n]\setminus \eta$, then 
\begin{equation*}
F_{Q_1} / (F_{Q_1} \cap \kk^{\mu}) = F_{Q_2} \cap \kk^{\eta} \hspace{15pt} \text{ and } \hspace{15pt} F_{Q_2} / (F_{Q_2} \cap \kk^{\mu}) = F_{Q_1} \cap \kk^{\eta}
\end{equation*}
under the identifications $\kk^{n}/\kk^{\mu} = \kk^{\eta}$ and $\kk^{n}/\kk^{\eta} = \kk^{\mu}$.

In the construction of a morphism to $\varprojlim_{Q\in \Delta_w} \Gr_{M_Q}$, it suffices to construct compatible morphisms $\inw_w\GrU{d}{n} \to \Gr_{M_Q}$ whenever $Q$ is an internal cell of $\Delta_w$. Let us sketch a geometric characterization of these morphisms. Choose a $\kk$-point  $x$ of  $\inw_w\GrU{d}{n}$ and set $\KK = \kk(\!(t^{\RR})\!)$.  By surjectivity of exploded tropicalization \cite{Payne}, there is a $\KK$-point $p$ of $\GrU{d}{n}$ such that $\ETrop(p) = x$; because $\KK$ is a generalized power series field, $\ETrop(p)$ is simply the vector of lead coefficients. The Pl\"ucker vector $p$ defines a linear subspace  $F_p$ of $(\KK^*)^n$ whose tropicalization is $L_w$. For any $v$ in the bounded cell of $L_w$ dual to $Q$, the closure of $\inw_v(F_p)$ in $\kk^n$ is a linear realization of $M_Q$. The morphism $\inw_w\GrU{d}{n} \to \Gr_{M_Q}$ sends $x$ to $\inw_v(F_p)$. We will produce a scheme-theoretic construction of these morphisms in \Cref{sec:limitsTSC}, and provide compatibility with this geometric description in \Cref{rmk:indegTSCGeometric}.

Our main application of \Cref{thm:closedEmbeddingIntro} is to determine smoothness and irreducibility of initial degenerations of Grassmannians, especially for $\GrU{3}{7}$. Because $\TGrU{d}{n}$ is sensitive to the characteristic of the underlying field, we assume that $\charF \kk = 0$.

\begin{theorem}
	\label{theoremA} 
	The initial degenerations of $\GrU{3}{7}$ are smooth and irreducible.
\end{theorem}

 The computation of $\TGrU{3}{7}$ in \cite{HJJS09} allows us to compute the initial degenerations of $\GrU{3}{7}$ and matroid subdivisions of $\Delta(3,7)$.   Given the size of the initial ideals, determining smoothness and irreducibility of the $\inw_w\GrU{3}{7}$ directly is impractical, even with computer assistance. In comparison, thin Schubert cells and the morphisms between them are easier to describe, as we do in Sections \ref{sec:smoothnessOfTSC} and \ref{sec:smoothnessMorphisms}. Each $\Gr_M\subset \GrC{3}{7}$ is smooth and irreducible and the morphisms $\Gr_M\to\Gr_{M_Q}$ are  smooth and dominant with connected fibers, provided $Q$ is not a face of $\Delta(3,7)$.  This allows us to determine smoothness and irreducibility of $\varprojlim_{Q\in\Delta_w}\Gr_{M_Q}$; see Examples \ref{ex:limitTree} and \ref{ex:affineCoordinates} for an illustration of this analysis. 
Being a closed immersion of affine schemes, $\psi_w:\inw_w\GrU{d}{n} \hookrightarrow \varprojlim_{Q\in \Delta_{w}}\Gr_{M_Q}$ is an isomorphism whenever $\varprojlim_{Q\in\Delta_w}\Gr_{M_Q}$ is an integral scheme of dimension $d(n-d)$. While the inequality $\dim( \varprojlim_{Q\in\Delta_w}\Gr_{M_Q}) \geq d(n-d)$ may be strict when $d=3$ and $n\geq 9$, as demonstrated in \Cref{ex:39counterexample}, it is an equality for all $w$ in the $(3,7)$ case. This will yield a proof of \Cref{theoremA}, and our techniques will allow us to prove an analog of this theorem for any  $\Gr_M\subset \GrC{d}{n}$ for $(2,n)$, $(3,6)$, and $(3,7)$, see \Cref{thm:initDegTSC}.

As a consequence of \Cref{theoremA}, $\GrU{3}{7}$ is sch\"on in the sense of Tevelev \cite{Tevelev2007}.  This is important because, when $X_0$ is a sch\"on subvariety of a torus, we may use tropical geometry to construct compactifications of $X_0$ with desirable properties.  Indeed, the closure $X$ of $X_0$ in any toric variety whose fan has support $\Trop X_0$ is a sch\"on compactification \cite{LuxtonQu}. The strata of $X$ are sch\"on, and $(X,B:=X\setminus X_0)$ has toroidal singularities.  Hacking, Keel, and Tevelev prove that $K_X + B$ is ample if and only if each irreducible  stratum of $X$ is log minimal, and a sch\"on subvariety of a torus  is log minimal if and only if its tropicalization is not invariant under translation by a rational subspace \cite{HackingKeelTevelev2009}.
They apply this to $Y^n$,  the moduli space of smooth marked del Pezzo surfaces of degree $9-n$ for $n \leq 7$, demonstrating that the Sekiguchi cross-ratio variety $\overline{Y}^n$ \cite{Sekiguchi1, Sekiguchi2}, introduced by Naruki when $n=6$ \cite{Naruki}, is a sch\"on and log canonical compactification of $Y^n$.

While $\GrU{d}{n}$ is not log minimal, its quotient $\XU{d}{n}$ by the free action of the maximal torus $H\subset \PGL(d)$ does have this property \cite[Proposition~2.20]{KeelTevelev2006}.  Via the Gelfand-MacPherson correspondence, we interpret $\XU{d}{n}$ as the moduli space of $d$ marked points in $\PP^{n-1}$ in linear general position up to the  $\PGL(d)$-action.  
The Chow quotient  $\GrC{d}{n}/\!/ H$ compactifies $\XU{d}{n}$. Let $\XC{d}{n}$ be its normalization. Kapranov \cite{Kapranov1993} proves $\XC{2}{n} \cong \overline{M}_{0,n}$, the Grothendieck-Knudsen moduli space of genus $0$ stable $n$-marked curves. This compactification of $\XU{2}{n}$ is sch\"on \cite{Tevelev2007} and log canonical \cite{KeelMcKernan}.  Keel and Tevelev prove $\XC{3}{n}$ is not log canonical when $n\geq 9$, and together with Hacking they conjecture $\XC{3}{n}$ is a sch\"on and log canonical compactification for $\XU{3}{n}$ when $n=6,7$, and $8$ \cite[Theorem~5.7]{Tevelev2007}, \cite[Conjecture~1.6]{KeelTevelev2006}. 
Luxton handles the $n=6$ case by investigating the relationship between $\XU{3}{6}$ to $Y^6$ \cite{Luxton2008}. He proves that $\XU{3}{6}$ is sch\"on by showing that the toric strata of $\XC{3}{6}$ are smooth via a delicate analysis of how the log canonical fan of $Y^6$ maps onto $\Trop \XU{3}{6}$. A direct adaptation of Luxton's strategy does not carry over to this setting; see \Cref{rmk:Luxton}.   Instead, we determine that $\XU{3}{7}$ is sch\"on directly from \Cref{theoremA}. We use this to verify the above conjecture when $n=7$. 
\begin{theorem}
	\label{theoremB}
	The variety  $\XC{3}{7}$ is a sch\"on and log canonical compactification of $\XU{3}{7}$. 
\end{theorem}

In \Cref{sec:higherGrassmannians}, we investigate the behavior of initial degenerations of $\GrU{3}{n}$ for larger values of $n$. 
Given the relationship between thin Schubert cells and initial degenerations of $\GrU{d}{n}$, it is reasonable to expect that in general $\GrU{d}{n}$ will have initial degenerations that are not smooth or reducible. Indeed, the Perles matroid (see  \Cref{fig:higherGrassmannians})  is a rank 3 matroid $P$ on $[9]$ such that $\Gr_{P}$ is reducible. 
We use this to find an initial degeneration of $\GrU{3}{9}$ with the same property.  
\begin{theorem}
	\label{theoremC}
	The Grassmannian $\GrU{3}{9}$ has an initial degeneration with two connected components.  
\end{theorem}

We conclude  with three appendices. \Cref{appendix:SDC} gathers various properties of morphisms that are smooth and dominant with  connected fibers. It also includes a discussion on finite limits of $\kk$-schemes.  \Cref{appendix:37Data} contains a table with the data necessary for the proof of \Cref{lm:centerlimit}.  \Cref{appendix:MAC}, written by Mar\'ia Ang\'elica Cueto, includes various arguments that reduce the study of thin Schubert cells $\Gr_{M}$, and the morphisms $\Gr_{M} \to \Gr_{M'}$ to the case where $M$ is a simple and connected matroid, and a treatment of the rank 2 case.  It also includes an argument that the limits of thin Schubert cells over $\Delta_{w}$ may be computed on the smaller poset consisting of cells that have codimension at most one.

\subsection*{Conventions} 
\noindent Throughout, $\kk$ will denote an algebraically closed field of characteristic $0$. We fix the assumption on the characteristic because of the dependence on computer calculations.    However, the proof of \Cref{thm:closedEmbeddingIntro} works for all characteristics, and we expect that \Cref{theoremA} and \Cref{theoremC} remain true provided $\charF \kk\neq 2 $ or $5$ respectively.

\subsection*{Computations} 
\noindent The software packages~\gfan~\cite{gfan}, ~\MT~
\cite{M2},~\pmk~\cite{polymake:2000}, and~\sage~\cite{sagemath} in the proofs of \Cref{prop:3nDimension},  \Cref{lm:centerlimit}, and \Cref{lem:raysOrigin}. The matroid subdivisions in Examples \ref{ex:limitTree}, \ref{ex:affineCoordinates}, \ref{ex:37counterexample},  \ref{ex:39counterexample} and \Cref{pr:Perles2} are computed using~\pmk. No computation takes longer than a few hours on a standard desktop computer. The code may be found at the following website. 
\begin{center}
	\url{https://github.com/dcorey2814/initialDegenerationsOfGrassmannians}
\end{center}

\subsection*{Acknowledgments} I am especially grateful to my Ph.D. advisor Sam Payne for his guidance throughout the course of this project,  and for his feedback on earlier versions on this paper.  Some arguments in \Cref{sec:smoothnessMorphisms}, especially the proofs of Lemmas \ref{lem:smooth} and \ref{lem:smooth2}, are improvements of those found in an earlier version of this paper \cite{Corey17}, inspired by conversations with Mar\'ia Ang\'elica Cueto.  I also express my gratitude to her for contributing \Cref{appendix:MAC}. 
I thank Dustin Cartwright, Netanel Friedenberg, Dhruv Ranganathan, Yue Ren, Jenia Tevelev, and Jeremy Usatine for many helpful conversations, as well as David Jensen for comments on an earlier draft. I would like to extend my gratitude to the hospitality of the Fields Institute for Research in Mathematical Sciences and the organizers of the Major Thematic Program on Combinatorial Algebraic Geometry (July-December 2016) where this project began. This research was partially supported by NSF grant CAREER DMS-1149054 (PI: Sam Payne) and NSF RTG Award DMS–1502553.

\section{Preliminaries}\label{sec:Preliminaries}
\subsection{Initial Degenerations}
\noindent We recall some facts about initial degenerations and tropicalization of varieties defined over a trivially valued field from the Gr\"obner-theoretic perspective, see  \cite[Chapters~2,~3]{MaclaganSturmfels2015} for a comprehensive treatment, including the non-trivially valued case. Let $X$ be the closed subvariety of $\PP^{a} = \Proj(\kk[t_0,\ldots, t_a])$ with homogeneous ideal $I\subset \kk[t_0,\ldots, t_a]$. Assume that $X$ meets the dense torus $T$.  Set $X_0 = X\cap T$ and $I_0 = I\cdot \kk[t_0^{\pm},\ldots,t_a^{\pm}]$. We will often find it easier to work with $\Spec(\kk[t_0^{\pm},\ldots,t_a^{\pm}]/I_0)$. This space is $\pi^{-1}(X_0)$, where $\pi:\AA^{a+1}\setminus \{0\}\to \PP^a$ is the natural projection. Note that $\pi^{-1}(X_0) \cong X_0\times \GGm$.

Let  $N\cong \ZZ^{a+1}/\zspan{\mathbf{1}}$  denote the cocharacter lattice of $T$, where $\mathbf{1} = (1,\ldots,1)$. For $z = (z_0, \ldots, z_a)$, we write $t^{z} = t^{z_0}\cdots t^{z_a}$. The \textit{initial form} of $f\in \kk[t_0^{\pm}, \ldots, t_a^{\pm}]$ with respect to $w\in (N_{T})_{\RR}:=(N_{T})\otimes_{\ZZ} \RR$ is
\begin{equation*} 
\inw_w f = \sum_{w:\left\langle w,z \right\rangle \text{ minimal}} a_zt^z \hspace{10pt} \text{ where } \hspace{10pt} f = \sum a_z t^z. 
\end{equation*} 
That is, $\inw_wf$ is the sum of all monomials $a_zt^z$ of $f$ with minimal $w$-weight.  The \textit{initial ideals} of $I_0$ and  $I$ with respect to $w$ are
\begin{equation*} 
\inw_w I_0 = \langle \inw_wf \, | \, f\in I_0 \rangle  \hspace{10pt} \text{ and } \hspace{10pt} \inw_w I = \langle \inw_wf \, | \, f\in I \rangle,
\end{equation*} 
respectively. The \textit{initial degeneration} of $X_0$ with respect to $w$ is  
\begin{equation*}
\inw_wX_0 = T\cap \Proj(\kk[t_0, \cdots, t_a]/\inw_w I).
\end{equation*}

There is a complete polyhedral fan $\Sigma_{\Grob}(X_0)$ in  $N_{\RR}$, called the \textit{Gr\"obner fan}, where $w$ and $w'$ belong to the relative interior of the same cone in $\Sigma_{\Grob}(X_0)$ if and only if $\inw_w I = \inw_{w'} I$  \cite[Theorem~1.2]{Sturmfels1996}.   The \textit{tropicalization} of $X_0$ is  
\begin{equation*} 
\Trop X_0 = \left\lbrace w\in (N_T)_{\RR} \ | \ \inw_w I_0 \neq \left\langle 1 \right\rangle   \right\rbrace.  
\end{equation*}    
When $X_0$ is irreducible,  $\Trop X_0$ is the support of a pure $\dim (X_0)$-dimensional subfan of $\Sigma_{\Grob}(X_0)$. Denote the restriction of  $\Sigma_{\Grob}(X_0)$ to  $\Trop X_0$ by $\cG_{X_0}$. While $\inw_wX_0$ depends only on the cone of $\cG_{X_0}$ containing $w$ in its relative interior, it is possible that $\inw_wX_0 = \inw_{w'}X_0$ when $w$ and $w'$ belong to distinct locally closed cones. In this case, $\inw_wI$ and $\inw_{w'}I$ differ by primary components contained in $\langle t_0,\ldots,t_a\rangle$. 

\subsection{Matroid polytopes} \label{sec:matroids}
\noindent We assume familiarity with basic notions of matroids and refer the reader to \cite{Oxley1992} for a detailed treatment. For brevity, we say that a rank $d$ matroid on $[n]$ is a $\mat{d}{n}$-matroid. Given a matroid $M$, we write $\cB(M)$ for its set of bases and $\rho_M$ for its rank function. The uniform $\mat{d}{n}$-matroid is denoted by $U(d,n)$. For $\eta\subset [n]$, $M/\eta$ denotes the contraction of $M$ by $\eta$ and $M|\eta$ denotes the restriction of $M$ to $\eta$.

Let $e_0,\ldots, e_{n-1}$ denote the standard basis of $\RR^n$, and for a subset $\lambda = \{\lambda_0, \ldots, \lambda_{k}\} $ of $[n]$, let $e_{\lambda} =  e_{\lambda_0} + \cdots + e_{\lambda_{k}}.$ 
The \textit{hypersimplex} $\Delta(d,n)$ is the polytope in $\RR^n$ defined by
\begin{equation} 
\label{eq:hypersimplex}
\Delta(d,n) = \left\lbrace  (x_0, \ldots, x_{n-1}) \in \RR^n \, \left| \, x_{[n]} = d,\ 0\leq x_i\leq 1 \right.\right\rbrace.
\end{equation} 
The vertices of $\Delta(d,n)$ are  the points $e_{\lambda}$ for $\lambda\in {[n]\choose d}:=\{\sigma\subset [n]\, | \, |\sigma|=d \}$. The \textit{matroid polytope} of $M$ is
\begin{equation} 
\label{eqn:matroidpolytope}
Q_M = \left\lbrace  (x_0, \ldots, x_{n-1}) \in \RR^n \, \left| \,   x_{[n]} = d, \ x_{\eta} \leq \rho_M(\eta),\ \eta \subset [n]\right.\right\rbrace.
\end{equation} 
The vertices of $Q_{M}$ are the points $e_{\beta}$ for $\beta \in \cB(M)$. Given a collection of vertices of $\Delta(d,n)$, its convex hull $Q$  is a matroid polytope if and only if every edge of $Q$ is parallel to some $e_i-e_j$ \cite[Theorem~4.1]{GGMS}; we write $M_Q$ for the corresponding matroid. In particular, any face of a matroid polytope is a matroid polytope.

Throughout, we will consider the face order on polytopes: $Q'\leq Q$ whenever $Q'$ is a face of $Q$, and $Q'\lessdot Q$ when $Q'$ is a facet of $Q$. This induces a partial order on the set of $\mat{d}{n}$-matroids: $M'\leq M$ whenever $Q_{M'}\leq Q_{M}$, and $M'\lessdot M$ if $Q_{M'} \lessdot Q_{M}$. Given $\eta\subset [n]$, let $M_{\eta} = M_{Q'}$ where  $Q'=Q_M\cap \{x_{\eta} = \rho_{M}(\eta)\}$. The bases of $M_{\eta}$ are
\begin{equation*}
\cB(M_{\eta}) = \{\beta\in \cB(M) \, | \, |\beta \cap \eta| = \rho_M(\eta)  \},
\end{equation*}
and the remaining $ \beta \in \cB(M) \setminus \cB(M_{\eta})$ satisfy $|\beta\cap\eta| < \rho_M(\eta)$. It is not hard to produce an isomorphism $M_{\eta} \cong  M/\eta \oplus M|\eta$. 
When $M$ is connected, a nonempty subset $\eta$ is \textit{nondegenerate} if $M/{\eta}$ and $M|\eta$ are connected. The following proposition may be found in \cite[Section~2.5]{GelfandSerganova}.
\begin{proposition}
	\label{prop:facesMatroidPolytope}
	If $M$ be a connected matroid on $[n]$, then $\eta \mapsto Q_{M_{\eta}}$ is a one-to-one correspondence between nondegenerate subsets $\eta$ and the facets of $Q_M$.
\end{proposition}

\noindent Finally, we remark that  if $w\in \ZZ^n$ and $M_w$ is the matroid of minimal $w$-weight as in \cite{ArdilaKlivans}, then our $M_{\eta}$ is just $M_{-\chi(\eta)}$ where $\chi$ is the characteristic function.

\subsection{Thin Schubert cells} 
The Grassmannian $\GrC{d}{n}$ of $d$-dimensional linear subspaces of $\kk^n$ is a subvariety of $\PP(\wedge^d\kk^n)$ via the Pl\"ucker embedding. We write $\kk[p_{\lambda}]$ for the homogeneous coordinate ring of $\PP(\wedge^d\kk^n)$, $I^{d,n}\subset \kk[p_{\lambda}]$ for the Pl\"ucker ideal, and $p_{\lambda}(F)$ for the $\lambda$-th Pl\"ucker coordinate of $F\in \GrC{d}{n}$. As observed by \cite{GGMS}, $\GrC{d}{n}$ decomposes into locally closed subschemes $\Gr_M$ called \textit{thin Schubert cells} which are indexed by $\kk$-realizable $\mat{d}{n}$-matroids. 
 Set-theoretically, 
\begin{equation*} 
\Gr_M = \{F\in \GrC{d}{n} \, |\, p_{\lambda}(F) \neq 0 \text{ if and only if } \lambda \in \cB(M)\}.
\end{equation*} 
Observe that $\GrU{d}{n} = \Gr_{U(d,n)}$. We realize $\Gr_M$ as a scheme in the following way. Define 
\begin{itemize}[noitemsep]
	\item $B_M = \kk[p_{\lambda} \, | \, \lambda \in \cB(M)] \subset \kk[p_{\lambda}]$,
	\item $I_{M} = \left(I^{d,n} + \left\langle p_{\lambda}\, |\, \lambda \in {[n]\choose d} \setminus  \cB(M)\right\rangle \right) \cap B_M$, 
	\item $S_M$ the multiplicative semigroup of $B_M$ generated by $p_{\lambda}$ such that $\lambda \in \cB(M)$, and
	\item $R_M = S_M^{-1}B_M/I_M$.
\end{itemize}
Then 
\begin{equation*}
\Gr_M = T(M) \cap \Proj(B_M/I_M)
\end{equation*}
where $T(M)$ is the dense torus of $\Proj(B_M)$. For computations, we will often find it easier to work with $\Spec(R_M) \cong \Gr_M \times \GGm$. The ideal $I_M$ is generated by 
\begin{equation} 
\label{eqn:pluckergenerator}
\qp{M}{\mu}{\nu} = \sum_{i:\mu \cup i, \nu\setminus i \in \cB(M) } \sgn(i;\mu,\nu) p_{\mu \cup i} p_{\nu\setminus i} 
\end{equation}
where $\mu \in {[n]\choose d-1}$ is independent and $\nu \in {[n]\choose d+1}$, $\nu\not\subset\mu$ has rank $d$ \cite[Equation~4.4.1]{MaclaganSturmfels2015}. Here, $\sgn(i;\mu,\nu)$ equals $(-1)^{\ell}$ where $\ell$ is the number of $j\in \nu$ with $i<j$ plus the number of elements $j'\in \mu$ such that $i>j'$. 
The coordinate ring of $\Gr_M$ can be presented with far fewer generators and relations by using affine coordinates with respect to a fixed basis, which we now describe.

\begin{construction}
	\label{construction}
	Suppose $M$ is a $\kk$-realizable $\mat{d}{n}$-matroid. Let $\beta = \{b_0<\cdots<b_{d-1}\}$ be a basis, $\gamma = \{c_0<\ldots<c_{n-d-1}\}$ its complement, and $\kk[x_{ij}]:=\kk[x_{ij}\ |\ 0\leq i < d,\; 0\leq j < n-d]$. Define a matrix $X$ in the following way. The submatrix of $X$ formed by the columns from $\beta$ is the identity matrix, and the submatrix formed by the  columns from $\gamma$ has $(i,j)$-entry equal to $x_{ij}$. For example, if $\beta = [d]$, then 
	\begin{equation*} X = \begin{pmatrix}
	1 & 0 & \cdots &0  & x_{00} & x_{01} &\cdots & x_{0,n-d-1}  \\
	0 & 1 & \cdots &0 & x_{10} & x_{11} & \cdots & x_{1,n-d-1} \\
	\cdots & \cdots & \cdots &\cdots & \cdots & \cdots & \cdots & \cdots \\
	0 & 0 & \cdots& 1 & x_{d-1,0} &x_{d-1,1} & \cdots & x_{d-1,n-d-1} \\
	\end{pmatrix}. \end{equation*} 
	Given $\lambda\in {[n]\choose d}$, let  $X_{\lambda}$ be $d\times d$ the minor of $X$ formed by the columns from $\lambda$. For $i \in [d]$ and $j\in [n-d]$, define $\lambda_{ij} \in {[n]\choose d}$ by 
	\begin{equation*} \lambda_{ij}  = (\beta\setminus \{b_i\})\cup\{c_j\}. \end{equation*} 
	Then $x_{ij} = (-1)^{\ell} X_{\lambda_{ij}}$ where $\ell$ is the number of elements of $\beta$ strictly between $b_i$ and $c_j$. We define
	\begin{itemize}[noitemsep]
		\item $B_M^x  = \kk[x_{ij} \, | \, \lambda_{ij} \in \cB(M)] \subset \kk[x_{ij}]$, 
		\item $I_M^x = \left\langle X_{\lambda} \, | \, \lambda \in {[n]\choose d} \setminus \cB(M) \right\rangle \cap B_M^x$, and
		\item $S_M^{x}$ the multiplicative semigroup in $B_M^x$ generated by $\overline{X}_{\lambda}:=\pi_M(X_{\lambda})$ for $\lambda\in \cB(M)$, where $\pi_M:\kk[x_{ij}] \to \kk[x_{ij}] / \langle x_{ij} \, | \, \lambda_{ij} \notin \cB(M) \rangle \cong B_{M}^x$ is the quotient map. 
	\end{itemize}
	Then the coordinate ring of $\Gr_M$ is isomorphic to  $R_M^{x} :=(S_M^{x})^{-1} B_{M}^x/I_{M}^{x}$.
\end{construction}

Thin Schubert cells behave well with respect to duality and direct sum of matroids. If $M^{*}$ is the dual of $M$, then $\Gr_{M^*} \subset \GrC{n-d}{n}$, and $\Gr_M \cong \Gr_{M^{*}}$ under the isomorphism $\GrC{d}{n} \cong \GrC{n-d}{n}$. If $M$ decomposes as $M = M_1\oplus M_2$, then $\Gr_M \cong \Gr_{M_1} \times \Gr_{M_2}$ \cite[Proposition~9.4]{KatzMatroid}.

\subsection{Matroid subdivisions and the tropical Grassmannian} \label{sec:TropGrassmannian}
\noindent Throughout, we will use the following abbreviations: $\TGr_M = \Trop \Gr_M$, $\cG_M = \cG_{\Gr_{M}}$, $\TGrU{d}{n} = \Trop \GrU{d}{n}$ and $\cG_{d,n} = \cG_{\Gr_{U(d,n)}}$.  
Given a polytope $P\subset \RR^n$ with vertices $u_{0},\ldots,u_k$ and $w\in \RR^{k+1}$, define
\begin{equation*} 
P_{w} = \conv\{(u_i, w_{i}) \, |\, 0\leq i\leq k  \}.
\end{equation*}
Any lower face of $P_{w}$ is of the form 
\begin{equation*}
\label{eq:lowerFaceDeltaM}
\face_{v}(P_w) = \{ x \in P_{w}  \, | \, \langle x, \mathbf{v} \rangle \leq \langle y,\mathbf{v} \rangle \text{ for all } y\in P_{w}   \}
\end{equation*}
where $v\in \RR^n$ and  $\mathbf{v} = (v,1)$. The lower faces of $P_{w}$ project onto $P$, forming a polyhedral complex  whose support is $P$. This is called the \textit{regular subdivision} of $P$ induced by $w$. The \textit{secondary fan}  $\SF{P}$ of $P$ is the complete fan in $\RR^{k+1}$ where $w$ and $w'$ belong to the relative interior of the same cone if and only if they induce the same regular subdivision on $P$ \cite{GKZ}. 
The \textit{adjacency graph} of this subdivision is the graph with vertex $v_Q$ for each maximal cell $Q$ and an edge between $v_{Q}$ and $v_{Q'}$ whenever $Q$ and $Q'$ share a common facet.

Given a $\mat{d}{n}$-matroid $M$ and $w\in \RR^{\cB(M)}$, we write $\Delta_{M,w}$ for the regular subdivision of $Q_M$ induced by $w$.  This subdivision is \textit{matroidal}, or $\Delta_{M,w}$ is a \textit{matroid subdivision},  if  each $Q\in\Delta_{M,w}$ is a matroid polytope. The \textit{Dressian} of $M$ is the subfan of $\SF{Q_M}$ defined by
\begin{equation*}
\Dr_M = \left\{w\in \RR^{\cB(M)} \, | \, \Delta_{M,w} \text{ is matroidal} \right\}.
\end{equation*}
We write $\Delta_w = \Delta_{U(d,n),w}$ and  $\DrU{d}{n}= \Dr_{U(d,n)}$. 
Let $\Delta_{M,w}$ be a matroid subdivision, $Q\in \Delta_{M,w}$ determined by $\face_{v}(Q_M)$, and  $u\in \RR^{\cB(M)}$  the vector with coordinates $u_{\lambda} = \langle v,e_{\lambda} \rangle + w_{\lambda}$. Then 
\begin{equation}
\label{eq:basesCell}
\cB(M_Q) = \{\lambda\in \cB(M) \, | \, u_{\lambda} \leq u_{\lambda'} \text{ for all } \lambda'\in \cB(M) \}.
\end{equation}

If $w\in \TGr_M$, then $\Delta_{M,w}$ is  matroidal \cite[Lemma~4.4.6]{MaclaganSturmfels2015}. In fact,
the inclusion $\TGr_M\subset\Dr_M$ is a morphism of fans (when $M$ is the uniform matroid, this is \cite[Theorem~5.4]{Tevelev2007}), thus $\inw_w\Gr_M$ and $\Delta_{M,w}$ depend only on the cone of $\cG_M$ containing $w$ in its relative interior. We have $\TGrU{2}{n} = \DrU{2}{n}$ as fans, and $\TGrU{3}{6}$ is a refinement of $\DrU{3}{6}$ \cite{SpeyerSturmfels2004a}. Because $\charF \kk \neq 2$, $\TGrU{3}{7}$ is a refinement of a subfan of $\DrU{3}{7}$  \cite[Theorems~2.1, 2.2]{HJJS09}. 
Therefore, there is a coarser fan structure on $\TGrU{3}{n}$ when $n=6,7$, which we denote by  $\LFU{3}{n}$.

\section{Limits of thin Schubert cells} 
\label{sec:limitsTSC}

\noindent In this section, we construct closed immersions  
\begin{equation*}
\inw_w\Gr_M \hookrightarrow \varprojlim_{Q\in \Delta_{M,w}} \Gr_{M_Q}
\end{equation*}
for  any $\kk$-realizable matroid $M$ and $w\in \TGr_M$, proving \Cref{thm:closedEmbeddingIntro}. We begin with a discussion of the morphsms between thin Schubert cells. Let $Q\subset \Delta(d,n)$ be a matroid polytope, $Q'$ the face defined by the equation $x_{\eta} = \rho_{M_Q}(\eta)$, and $\mu=[n]\setminus \eta$. As discussed in the introduction, $F \mapsto (F\cap \kk^{\mu}) \oplus F/(F\cap \kk^{\mu})$ characterizes the map $\Gr_{M_Q} \to \Gr_{M_{Q'}}$ set theoretically. From the canonical isomorphism 
\begin{equation*}
\wedge^d F \cong \wedge^{d-\rho_{M_Q}(\eta)} (F \cap \kk^{\mu})
 \otimes \wedge^{\rho_{M_Q}(\eta)} F/(F\cap \kk^{\mu})   
\end{equation*}
we see that $\Gr_{M_Q} \to \Gr_{M_{Q'}}$ is induced by the projection $\kk^{\cB(M_Q)} \to \kk^{\cB(M_{Q'})}$ \cite[Proposition~I.6]{Lafforgue2003}. 
We derive a scheme-theoretic characterizations of these morphisms.

\begin{proposition}
	\label{prop:facemaps} 
	Suppose $M'\leq M$ are $\mat{d}{n}$-matroids. The inclusion $B_{M'} \subset B_M$ induces a 
	 morphism of schemes $\varphi_{M,M'}:\Gr_{M} \rightarrow \Gr_{M'}$. Furthermore, these morphisms satisfy $\varphi_{M,M''} = \varphi_{M',M''}\varphi_{M,M'}$  if $M''\leq M' \leq M$ and $\varphi_{M,M} = \id$.
\end{proposition}

\begin{proof}
	It suffices to consider the case  $M' = M_{\eta}$ for some $\eta\subset [n]$. 
	 We must show that $I_{M_{\eta}}$ maps to $I_M$ under the inclusion $B_{M_{\eta}}\subset B_M$. We will do this using the  generators for $I_{M_{\eta}}$ and $I_M$ given by \Cref{eqn:pluckergenerator}.  Suppose $\mu \in {[n]\choose d-1}$ is independent in $M_{\eta}$, and $\nu \in {[n]\choose d+1}$ not containing $\mu$ such that  $\rho_{M_{\eta}}(\nu) = d$. Note that $\mu$ is independent in $M$ and $\rho_M(\nu) = d$ because $\cB(M_{\eta}) \subset \cB(M)$.  We must show $\qp{M_{\eta}}{\mu}{\nu} = 0$ or $\qp{M}{\mu}{\nu}$. 
	  If $\qp{M_{\eta}}{\mu}{\nu} \neq 0$, then there is a $i_0\in \nu\setminus \mu$ such that both $\mu \cup i_0$ and $\nu\setminus i_0$ are in $\cB(M_{\eta})$, thus 	$ |(\mu \cup i_0) \cap \eta| $ and $ |(\nu \setminus i_0) \cap \eta| $ both equal  $r:=\rho_{M}(\eta)$. In particular,
	\begin{enumerate}[noitemsep]
		\item $|\mu \cap \eta| = r-1 $ and $|\nu \cap \eta| = r+1$ if  $i_{0}\in \eta$, or
		\item $|\mu \cap \eta| = r $ and $|\nu \cap \eta| = r$ if  $i_{0}\notin \eta$.
	\end{enumerate}

	\noindent For each $i \in \nu\setminus \mu$, we must show that $\mu \cup i$ and $\nu\setminus i $ are in $\cB(M_{\eta})$ if and only if they are both in $\cB(M)$. Since $\cB(M_{\eta}) \subset \cB(M)$, we need only show the ``if'' direction. 
	
	Suppose $\mu \cup i$ and $\nu \setminus i$ are bases of $M$. By the characterization of $Q_M$ in \Cref{eqn:matroidpolytope}, 
	\begin{equation}
	\label{eqn:lambdamui} 
	|(\mu \cup i) \cap \eta| \leq r\text{ and }  |(\nu \setminus i) \cap \eta| \leq r \end{equation}
	We show that they both equal $r$ by considering the possibilities of $|\mu\cap\eta|$ and $|\nu\cap\eta|$ as above. If $i_0 \in \eta$, then $|\mu \cap \eta| = r-1 $ and $|\nu \cap \eta| = r+1$. By \Cref{eqn:lambdamui}, we have that $|(\nu\setminus i)\cap\eta|  = r$. In particular, $i\in \eta$, so $|(\mu\cup i)\cap\eta|  = r$. If $i_0 \notin \eta$, then $|\mu \cap \eta| = r $ and $|\nu \cap \eta| = r$.  By \Cref{eqn:lambdamui}, we have that $|(\mu\cup i)\cap\eta|  = r$. In particular, $i\notin \eta$, so $|(\nu \setminus i) \cap \eta| = r$. 
\end{proof}

\begin{proposition}
	\label{prop:facemapsA}
	The induced morphism $\varphi_{M,M'}^{\#}:R_{M'}^x \to R_M^x$ is given by the inclusion  $B_{M'}^x \subset B_{M}^x$.
\end{proposition}

\begin{proof}
	Suppose $[d]$ is a basis of $M$ and $M'$. Setting  $\tilde{R}_{M} = S_{M}^{-1}\kk[p_{\lambda}/p_{[d]} \, | \, \lambda \in \cB(M)]/ I_M$, we see that $\theta_M:\tilde{R}_{M} \to R_M^x$
	given by $\theta_M(p_{\lambda}/p_{[d]}) = X_{\lambda}$ is an isomorphism (the inverse sends  $x_{ij}$ to $p_{\lambda_{ij}}/p_{[d]}$). By  \Cref{prop:facemaps}, $\varphi_{M,M'}$ is induced by the ring map $\psi_{M,M'}:\tilde{R}_{M'} \to \tilde{R}_M$ that sends $p_{\lambda}$ to itself.  Therefore,  $\theta_{M} \psi_{M,M'} \theta_{M'}^{-1}$ sends $x_{ij}$ to itself (for $\lambda_{ij}\in \cB(M)$), as required.
\end{proof}

Fix a $(d,n)$-matroid $M$  and $w\in \TGr_{M}$.  By  \Cref{prop:facemaps}, $\{\Gr_{M_{Q}} \, | \, Q\in \Delta_{M,w} \}$ defines an inverse system.  We may form $\varprojlim_{Q\in \Delta_{M,w}}\Gr_{M_{Q}}$, which we denote by $\Gr_{M,w}$, and write $\varphi_{Q}:\Gr_{M,w} \rightarrow \Gr_{M_Q}$ for the structure morphism.  When $M=U(d,n)$ we write $\Gr_w = \Gr_{U(d,n),w}$. Finite limits exist in the category of affine schemes because this category has fiber products and a terminal object \cite[Proposition~5.21]{Awodey}.

\begin{lemma}
	\label{lem:mapInitialDegToTSC}
	Suppose $w\in \TGr_{M}$ and $Q \in \Delta_{M,w}$. The inclusion $B_{M_Q} \subset B_M$ induces a 
	 morphism $\psi_{M,M_Q,w}:\inw_w\Gr_M \to \Gr_{M_Q}$. 
\end{lemma}

\begin{proof}
	Suppose $Q$ is determined by $\face_v(Q_{M})$. \Cref{eq:basesCell} records the bases  of $M_Q$.   We must show that $I_{M_Q}$ maps to $\inw_wI_M$ under the inclusion $B_{M_Q}\subset B_{M}$.  For this, it suffices to consider the quadratic generators $\qp{M_Q}{\mu}{\nu}$ from \Cref{eqn:lambdamui}. Let $\mu\in {[n] \choose d+1}$ with $\rho_{M_Q}(\mu) = d$ and $\nu\in {[n] \choose d-1}$ independent in $M_Q$. If $\qp{M_Q}{\mu}{\nu} \neq 0$, there  is a $i_{0}\in \mu\setminus \nu$ such that $\mu\setminus i_{0}, \nu\cup i_0 \in \cB(M_Q)$.  Because $\cB(M_Q)\subset \cB(M)$, $\mu,\nu,i_0$ satisfy the same properties for $M$. We must show
\begin{equation}
\label{eq:matroidInitialForm}
\qp{M_Q}{\mu}{\nu}  = \inw_w \qp{M}{\mu}{\nu}.
\end{equation}
Observe that for any $i,j\in \mu\setminus \nu$,  
\begin{equation}
\label{eq:uw}
u_{\mu\setminus j}  + u_{\nu\cup j} -  u_{\mu\setminus i}  - u_{\nu\cup i} = w_{\mu\setminus j}  + w_{\nu\cup j} - w_{\mu\setminus i} - w_{\nu\cup i}.
\end{equation}
where $u_{\lambda} = \langle v,e_{\lambda} \rangle + w_{\lambda}$. Now, $p_{\mu \setminus i}p_{\nu \cup i}$ is a summand in $\qp{M_Q}{\mu}{\nu}$ if and only if $u_{\mu \setminus i} = u_{\mu \setminus i_0}$ and $u_{\nu \cup i} = u_{\nu \cup i_0}$. By \Cref{eq:uw}, these equalities hold if and only if $w_{\mu\setminus i}  + w_{\nu\cup i} = w_{\mu\setminus i_0}  + w_{\nu\cup i_0}$. Since $w_{\mu\setminus i_0}  + w_{\nu\cup i_0}$ is the smallest such sum, we have this equality if and only if $p_{\mu \setminus i}p_{\nu \cup i}$ is a summand in $\inw_w\qp{M}{\mu}{\nu}$.	
\end{proof}

\begin{theorem}
	\label{thm:closedEmbedding}
	The morphisms $\psi_{M,M_Q,w}:\inw_w\Gr_M \to \Gr_{M_Q}$ induce a closed immersion $\psi_{M,w}: \inw_w\Gr_M \hookrightarrow \Gr_{M,w}$.
\end{theorem}

\begin{proof}
	Clearly $\varphi_{M_{Q},M_{Q'}}\psi_{M,M_Q,w} = \psi_{M,M_{Q'},w}$, so $\psi_{M,w}$ is defined by the universal property of $\Gr_{M,w}$. It is a closed immersion because the induced morphism 
 $\psi_{M,w}^{\#}: \varinjlim_{\Delta_{M,w}} R_{M'} \to S_{M}^{-1} B_{M}/\inw_wI_M$ is surjective.
\end{proof}

\noindent The following Corollary is an immediate consequence of \Cref{thm:closedEmbedding} and \Cref{prop:closedImmersionIso}. 

\begin{corollary}
	\label{cor:indegMSDIso}
	The closed immersion $\psi_{M,w}:\inw_w\Gr_M \hookrightarrow \Gr_{M,w}$ is an isomorphism when $\Gr_{M,w}$ is integral and of dimension $\dim\Gr_M$.
\end{corollary}

\begin{remark}
	\label{rmk:indegTSCGeometric}
	We now show that our definition of  $\psi_{M,M_Q,w}:\inw_w\Gr_M\to\Gr_{M_Q}$ agrees with the characterization discussed in the introduction.
	For simplicity, suppose $M=U(d,n)$.   We refer the reader to \cite[Chapter~4]{MaclaganSturmfels2015} for  basic facts about circuits of linear subspaces.	
	 As before, let $x$ be a $\kk$-point of $\inw_w\GrU{d}{n}$, $\KK = \kk(\!(t^{\RR})\!)$, and $p$  a $\KK$-point of $\GrU{d}{n}$ so that $\ETrop(p) = x$. The linear subspace $F_p \subset \Spec( \KK[y_0^{\pm},\ldots,y_{n-1}^{\pm}])$ with Pl\"ucker vector $p$ is defined by
	\begin{equation*}
	\ell_{\mu} = \sum_{k=0}^{d}(-1)^k p_{\mu\setminus i_k} \cdot y_{i_k}, \hspace{30pt} \mu = \{i_0, \ldots, i_{d}\} \in {\textstyle  {[n]\choose d+1}  }.
	\end{equation*}
	These form a universal Gr\"obner basis for the ideal they generate in $\KK[y_0,\ldots,y_{n-1}]$. Also, $L_w = \Trop F_p$. Choose a vector $v\in L_w$ in the cell dual to $Q$. Then $\inw_v(F_p)$ is cut out by the linear equations
	\begin{equation}
	\label{eq:initialLinearForms}
	\inw_v\ell_{\mu} = \sum_{k:u_{\mu\setminus i_k} \text{ minimal}} (-1)^k x_{\mu\setminus i_k} \cdot y_{i_k}
	\end{equation}	
	where  $u_{\lambda} =  \langle v,e_{\lambda}\rangle + w_{\lambda}$. Recall that the support of a linear form is $\ell = \sum a_iy_i$  is $\supp (\ell) = \{i\in[n] \, | \, a_i\neq 0    \}$.  The linear space $\inw_v(F_p)$ realizes a matroid $M'$ whose circuits are 
	\begin{equation*}
	\cC(M') = \{\supp (\inw_v\ell_{\mu}) \, | \, \mu \in \textstyle{{[n]\choose d+1}},\; \rho_{M}(\mu) = d \}
	\end{equation*}
	It is easy to see that $\cB(M') = \{\lambda\in{[n]\choose d} \, | \, u_{\lambda}\leq u_{\lambda'} \text{ for all } \lambda'\in \textstyle{{[n]\choose d}} \} $, hence $M_Q=M'$. By \Cref{eq:initialLinearForms} and the description of $\cC(M_Q)$, we see that the Pl\"ucker vector of $\inw_vF_\mu$ is the projection of $x\in \wedge^d\kk^n$ to $\kk^{\cB(M_Q)}$. This  is $\psi_{M,M_Q,w}(x)$, as required. 
\end{remark}

Now we show how to compute the coordinate ring of $\Gr_{M,w}$ in Pl\"ucker and affine coordinates. We will use this in \Cref{prop:3nDimension} below to compute the dimension of $\Gr_{M,w}$ for any $\kk$-realizable $(2,n)$, $(3,6)$, or $(3,7)$ matroid.  Let $\TC{M,w}$ be the collection of top dimensional cells in $\Delta_{M,w}$, and $\Gamma_{M,w}$ the adjacency graph of $\Delta_{M,w}$, as defined in \Cref{sec:TropGrassmannian}. For the uniform matroid, we write $\TC{w} = \TC{U(d,n),w}$ and $\Gamma_{w} = \Gamma_{U(d,n),w}$.
There is an inverse system over $\Gamma_{M,w}$ as in  \Cref{ex:diagramDualGraph}, and $\Gr_{M,w} \cong \varprojlim_{\Gamma_{M,w}} \Gr_{M'}$ by \Cref{pr:FromAllToCodim1Diagram}.
 Let
\begin{equation*}
I_{M,w} = \langle I_{M_Q}B_M \, |\, Q\in \TC{M,w} \rangle  \subset B_M
\end{equation*}
and $R_{M,w} = S_M^{-1}B_M/I_{M,w}$. When the  polytopes in $\TC{M,w}$ share a common vertex, let 
\begin{equation*}
I_{M,w}^x = \langle I_{M_Q}^xB_M^x \, |\, Q\in \TC{M,w} \rangle  \subset B_M^x
\end{equation*}
Given $f\in B_{M_Q}^x$, let $\overline{f} = \pi_{M_Q}(f)$ viewed as an element of $B_M^x$, where
$\pi_{M_Q}:\kk[x_{ij}] \to \kk[x_{ij}] / \langle x_{ij} \ | \ \lambda_{ij} \notin \cB(M_Q) \rangle \cong B_{M_Q}^x$ is the quotient map. 
Let $S_{M,w}^x$ be the multiplicative semigroup of $B_M^x$ generated by $\overline{X}_{\lambda}$ for each $\lambda\in \cB(M_Q)$  and   $Q\in \TC{M,w}$.  Finally, set $R_{M,w}^x = (S_{M,w}^x)^{-1}B_M^x/I_{M,w}^x$.

\begin{proposition}
	\label{prop:limitideal}
	For any $\mat{d}{n}$-matroid $M$ and $w\in \Dr_M$, 
	\begin{equation*}
	\Gr_{M,w} \cong \varprojlim_{\Gamma_{M,w}} \Gr_{M'} \cong T(M)\cap \Proj R_{M,w}
	\end{equation*}
	If the polytopes in $\TC{M,w}$ share a common vertex, then $\Gr_{M,w} \cong \Spec R_{M,w}^x$.
\end{proposition}

\begin{proof}
	The first isomorphism is established in \Cref{pr:FromAllToCodim1Diagram}.
	For each $Q\in \Delta_{M,w}$ of codimension 0 or 1, we have ring maps  $R_{M_Q} \rightarrow R_{M,w}$ and $R_{M_Q}^x \rightarrow R_{M,w}^x$  induced by $B_{M_Q}\subset B_{M}$ and $B_{M_Q}^x\subset B_{M}^x$ respectively. These produce  morphisms 
	\begin{equation*} 
	\Psi: \varinjlim_{\Gamma_{M,w}} R_{M_Q} \longrightarrow R_{M,w} \hspace{15pt} \text{ and } \hspace{15pt} \Psi^x: \varinjlim_{\Gamma_{M,w}} R_{M_Q}^x \longrightarrow R_{M,w}^x 
	\end{equation*}  
	Now let us construct inverses $\Theta$ and $\Theta^x$.  For $\lambda \in \cB(M)$ define $\Theta(p_{\lambda}) = \varphi_{M_Q}^{\#}(p_{\lambda})$ where  $Q\in \TC{M,w}$  and $\lambda \in \cB(M_Q)$. If $Q'$ is another such polytope, we must show that $\varphi_{M_Q}^{\#}(p_{\lambda}) = \varphi_{M_{Q'}}^{\#}(p_{\lambda})$. 
	When $Q'' = Q\cap Q'$ has codimension 1,   
	\begin{equation*}
	\varphi_{M_Q}^{\#}(p_{\lambda}) = \varphi_{M_{Q''}}^{\#}(p_{\lambda}) = \varphi_{M_{Q'}}^{\#}(p_{\lambda}).
	\end{equation*} 
	The general case follows from this observation and \Cref{lm:graphsAreConnected}.  Similarly, for $\lambda_{ij}\in\cB(M)$  define $\Theta^x(x_{ij}) = \varphi_{M_Q}^{\#}(x_{ij})$ where $Q\in \TC{M,w}$  and $\lambda_{ij} \in \cB(M_Q)$. It is easy to see that $\Theta$ and $\Theta^x$ take elements in $S_M$ and  $S_{M,w}^x$,  respectively, to invertible elements.

	Finally, we claim that $I_{M,w} \subset \ker(\Theta)$. It suffices to show that $\Theta(af) = 0$ for $a\in S_{M_Q}^{-1}B_{M_Q}$ and $f\in I_{M_Q}$ where $Q\in \TC{M,w}$. But $\Theta(af) = \Theta(a)\varphi_{M_Q}^{\#}(f) = 0$. This shows that $\Theta$ is defined on $R_{M,w}$. A similar argument shows that $I_{M,w}^x \subset \ker(\Theta^x)$. Therefore $\Theta$ and $\Theta^x$ are defined on $R_{M,w}$ and $R_{M,w}^x$ respectively. One may verify that they are inverses to $\Psi$ and $\Psi^x$ respectively.
\end{proof}

\begin{lemma}
	\label{prop:limitRank2}
	If $M$ is a rank $2$ matroid and $w\in \TGr_M$, then $\psi_{M,w}: \inw_w\Gr_M \to \Gr_{M,w}$ is an isomorphism.
\end{lemma}

\begin{proof}
By \Cref{thm:closedEmbedding} and  \Cref{prop:limitideal}, the identity on $B_M$ induces a surjective map $R_{M,w} \to S_{M}^{-1}B_{M}/\inw_wI_{M}$,
	so $I_{M,w} \subset \inw_wI_{M}$. The set 
	\begin{equation*}
	\cT = \left\{\qp{M}{\mu}{\nu} \, \left| \, |\mu| = 3, |\nu|=1, \mu\cap\nu = \emptyset \right.  \right\}
	\end{equation*}
	is a universal Gr\"obner basis for $I_M$ (when $M=U(2,n)$, this is the set of three-term Pl\"ucker relations). 
	Let $\qp{M}{\mu}{\nu}\in \cT$. If $Q\in \Delta_{M,w}$ such that $\qp{M_Q}{\mu}{\nu}\neq 0$, then $\qp{M_Q}{\mu}{\nu} = \inw_w\qp{M}{\mu}{\nu}$  by \Cref{eq:matroidInitialForm}, hence $I_{M,w} = \inw_wI_M$.
\end{proof}

\begin{proposition}
	\label{prop:3nDimension}
	Let $M$ be a $\kk$-realizable $(2,n)$, $(3,6)$, or $(3,7)$ matroid and $w\in \TGr_M$. Then $\dim \Gr_{M,w} = \dim \Gr_{M}$. 
\end{proposition}

\begin{proof}
	The rank $2$ case follows from \Cref{prop:limitRank2}, so let $M$ be a $\mat{3}{6}$ or $\mat{3}{7}$ matroid.  By \Cref{lm:simpleIsEnoughInitialDeg} we may assume that $M$ is simple.
	Once $\Delta_{M,w}$ is computed, this calculation may be done by hand, see \Cref{ex:limitTree} and \Cref{ex:affineCoordinates}. Due to the large number of cases, we use a computer. 
	The Gr\"obner fan structure on $\TGr_M$ is computed using~\gfan, and it catalogs all cones up to $\Aut(M)$-symmetry. The uniform cases were completed in \cite[Theorem~5.4]{SpeyerSturmfels2004a}  for $(3,6)$, and in \cite[Theorem~2.1]{HJJS09} for $(3,7)$. For each cone, we choose a representative weight vector $w$ and use~\pmk~to compute $\Delta_{M,w}$. 
	Let  $g$ be the product of all $p_{\lambda}$ for $\lambda\in \cB(M)$. Then $(I_{M,w}:g^{\infty})\subset B_M$ is the homogeneous ideal of the closure of $\Gr_{M,w}$ in $\Proj(B_M)$. We use~\MT~ to show that its dimension equals  $\dim \Gr_M$. The saturation was performed one variable at a time using the \texttt{saturate} function with the \texttt{Bayer} strategy. There are a total of 67 ideals to check among the simple $\mat{3}{6}$-matroids,  and 2815 ideals in the $\mat{3}{7}$ case (not counting $w=0$). The total process takes a couple minutes for $(3,6)$ and several hours for $(3,7)$.   
\end{proof}

\section{Geometry of thin Schubert cells}
\label{sec:smoothnessOfTSC}

By Mn\"ev universality, there exist $\mat{3}{n}$ matroids whose thin Schubert cells are singular or reducible, for sufficiently large $n$. Nevertheless,  $\Gr_M$ is smooth and irreducible  when $M$ is a rank 2 matroid, or a rank $3$ matroid on $[6]$ or $[7]$, as we demonstrate in this section.  Let $M$ be a $\kk$-realizable matroid. For each rank 1 flat $\eta$ of $M$, choose a non-loop $s_{\eta}\in \eta$, and set $S = \{s_{\eta} \,|\, \eta \text{ is a rank 1 flat} \}$. Then $M|S$ is a simple matroid, and $\Gr_M$ is the product of $\Gr_{M|S}$ with an algebraic torus as discussed in \Cref{lem:parallelTSC}. Because the only simple $\mat{2}{n}$-matroid is $U(2,n)$,  this leads to a straightforward proof that $\Gr_{M}$ is smooth and irreducible in the rank 2 case. Therefore, we will focus on rank $3$ matroids.  

Let $M$ be a $\kk$-realizable loop-free $\mat{3}{n}$-matroid for $n\geq 3$.  We can represent $M$ as a configuration of $n$ points $p_0, \ldots, p_{n-1}$ in $\PP^2$. The elements $i,j$ are parallel in $M$ if and only if  $p_i, p_j$ coincide. A subset $\beta \subset [n]$ is a basis if and only if $|\beta|=3$ and $p_i$ are not collinear for $i\in \beta$, and $\eta\subset [n]$ is a rank $2$ flat if and only if there is a line $L\subset \PP^{2}$ such that $p_{i} \in L$ precisely when $i \in \eta$. When drawing these pictures, we will only draw the points (labeled $0, \ldots, n-1$) and lines through at least 3 rank 1 flats, see Figures \ref{fig:affineCoordinatesEx} and \ref{fig:higherGrassmannians}. With this in mind, we say that $\eta$ is a \textit{line} of $M$ if $\eta$ is a rank $2$ flat and $|\eta\cap S|\geq 3$. The set of lines of $M$, denoted by $\cL(M)$, completely determines $M|S$.

All $\mat{3}{n}$-matroids for $3\leq n \leq 7$ (up to $S_n$-symmetry) can be found in the online \textit{Database of Matroids}

\begin{center} \url{http://www-imai.is.s.u-tokyo.ac.jp/~ymatsu/matroid/} \end{center}

\noindent In showing that $\Gr_M$ is smooth and irreducible for these matroids, we start with $n=3$ and work inductively.  At each step, we need only consider simple and connected matroids by Lemmas \ref{lem:reduceToConnected} and \ref{lem:parallelTSC}. However, there are still 8, resp. 21, simple and connected $\kk$-realizable $\mat{3}{6}$, resp. $\mat{3}{7}$-matroids. We use \Cref{lem:smoothScheme} to handle the remaining cases. 

In this section and the next, we will need the following definitions. A morphism of schemes is said to have \textit{connected fibers} if all of its nonempty fibers are connected. We say that $f:X\to Y$ is a \textit{SDC-morphism} if it is smooth and  dominant with connected fibers.

\begin{lemma}
	\label{lem:smoothScheme} 
	Suppose $M$ is a loop-free $\kk$-realizable $\mat{3}{n}$-matroid, $i\in [n]$  contained in exactly $k$ lines where $0\leq k \leq 2$, and  $\Gr_{M|[n]\setminus i}$ is integral.  The composition of a dominant open immersion $\Gr_M \hookrightarrow \Gr_{M|[n]\setminus i} \times \GGm^{3-k}$, followed by the projection away from $\GGm^{3-k}$ produces a SDC-morphism $\Gr_M \to\Gr_{M|[n]\setminus i}$.
\end{lemma}

\begin{proof}
	We use affine coordinates as in \Cref{construction}. Assume that $\{0,1,2\}$ is a basis of $M$, $i=n-1$, and the first 3 columns of $X$ form the identity matrix.   Suppose $n-1$ is not contained in any line.  This means that $\{i,j,n-1\} \in \cB(M)$ for $0\leq i<j\leq n-2$, so  $I_M^x$ is generated by $X_{\lambda}$ for suitable $\lambda\in{[n-1]\choose 3}$. Therefore  $R_M^x$ is obtained from $R_{M|[n-1]}^x[x_{0,n-4}^{\pm},x_{1,n-4}^{\pm},x_{2,n-4}^{\pm}]$ by inverting  $X_{\beta}$ for $\beta\in \cB(M)$. These ring elements are nonzero divisors since they are not 0 (by $\kk$-realizability) and $R_{M|[n-1]}$ is an integral domain. This localization produces the open immersion $\Gr_{M} \hookrightarrow \Gr_{M|[n-1]} \times \GGm^3$. 

Suppose $n-1$ is contained in exactly one line $\eta$. By applying a suitable permutation, assume $0,1\in\eta$.  Since $\lambda_{2,j-3}\notin \cB(M)$ when $j\in \eta$, $I_M^x$ is generated by $X_{\lambda}$ for suitable $\lambda\in{[n-1]\choose 3}$, and $R_M^x$ is obtained from $R_{M|[n-1]}^x[x_{0,n-4}^{\pm},x_{1,n-4}^{\pm}]$ by inverting the nonzero divisors  $X_{\beta}$ for $\beta\in \cB(M)$, producing the open immersion $\Gr_{M} \hookrightarrow \Gr_{M|[n-1]} \times \GGm^2$.

Now assume $n$ is contained in exactly two distinct lines $\eta_1$ and $\eta_2$. We may assume   $0,1\in \eta_1$ and $2\in \eta_2$. Because  $\lambda_{0,j-3}, \lambda_{1,j-3}\in\cB(M)$ when $j\in \eta_2 \setminus \{n-1\}$, the corresponding $x_{0,j-3}, x_{1,j-3}$ are invertible in $R_{M}^x$. Similar to the previous case, 
\begin{equation*}
R_{M}^x  =  (S_{M}^x)^{-1} R_{M|[n-1]}^x[x_{0,n-4}^{\pm}, x_{1,n-4}^{\pm}]/\langle x_{0,j-3}x_{1,n-4} - x_{1,j-3}x_{0,n-4} \, | \, j\in \eta_2\setminus \{2\} \rangle.
\end{equation*}
Since $|\eta_2| \geq 3$, this ring is isomorphic to $(S_{M}^x)^{-1}R_{M|[n-1]}^x[x_{1,n-1}^{\pm}]$, and we have an open immersion $\Gr_{M} \hookrightarrow \Gr_{M|[n-1]} \times \GGm$. Finally, $\Gr_{M} \to \Gr_{M|[n-1]} $ is a SDC-morphism by  \Cref{prop:SDCcomposition} and the fact that the projection away from $\GGm^{3-k}$ is SDC.  
\end{proof}

\begin{proposition}
	\label{prop:37SmoothIrreducible}
	For $3\leq n\leq 7$, $\Gr_{M}$ is smooth and irreducible for any $\kk$-realizable $\mat{3}{n}$-matroid $M$. 
\end{proposition}

\begin{proof}	
	The only $\mat{3}{3}$-matroid is $U(3,3)$, and its  thin Schubert cell consists of a single point. Next, of the four $\mat{3}{4}$-matroids up to $S_4$-symmetry,  $U(3,4)$ is the only one that is simple and connected. Since $\GrU{3}{4} \cong \GGm^3$, it is smooth an irreducible.  That the thin Schubert cells of the remaining three are also smooth and irreducible follows from  Lemmas \ref{lem:reduceToConnected}, \ref{lem:parallelTSC}, \Cref{prop:rk2smooth}, and the $\mat{3}{3}$-case.
	If $M$ is a $\mat{3}{5}$-matroid, then $M^*$ is a $\mat{2}{5}$-matroid, so  $\Gr_{M}$ is smooth and irreducible by  \Cref{prop:rk2smooth} and the isomorphism $\Gr_{M} \cong \Gr_{M^*}$. 
	
	Next, consider the $\mat{3}{6}$ case.  As before, we need only examine the simple and connected matroids. For every such matroid $M$, each $i\in [6]$ is contained in $2$ or fewer lines of $M$. Therefore, $\Gr_M$ is smooth and irreducible by  \Cref{lem:smoothScheme}, \Cref{prop:SDCOpen}(2), and the previous cases.  	Finally, if $M$ is any simple and connected $\mat{3}{7}$-matroid other than the Fano matroid, then $M$ has an $i\in [n]$ contained in no more than 2 lines. Similar to the $\mat{3}{6}$ case, $\Gr_M$ is smooth and irreducible. 
\end{proof}

\section{Morphisms between thin Schubert cells}
\label{sec:smoothnessMorphisms}

\noindent In this section, we will consider the  morphisms $\varphi_{M,M'}:\Gr_{M} \rightarrow \Gr_{M'}$.   
For arbitrary matroids $M$, SDC-properties of the morphisms $\varphi_{M,M'}$ are entirely determined by  $\varphi_{M|S,M'|S}$ with $S$ as in the beginning of \Cref{sec:smoothnessOfTSC}, hence we need only consider simple $M$. These reductions are contained in \Cref{appendix:MAC},  and yield a straightforward proof that $\varphi_{M,M'}$ is a SDC-morphism when $M$ is a $\mat{2}{n}$- matroid. 
Therefore, we focus on the rank $3$ case. For the proof of  \Cref{theoremA}, we will only need to verify that $\varphi_{M,M'}$ is a SDC-morphism for pairs $M'\lessdot M$ of $\mat{3}{7}$-matroids where $Q_{M'}$ is not a face of the hypersimplex. To do this, we will find it convenient to show that $\varphi_{M,M'}$ is a SDC-morphism for all pairs of $(3,m)$-matroids $M'\leq M$  where $m\leq 6$. Recall from \Cref{prop:facesMatroidPolytope} that the facets of $Q_M$ correspond to the nondegenerate subsets of $[n]$, when $M$ is connected. We begin with  a test for nondegeneracy in the rank 3 setting. 

\begin{proposition}
	\label{lm:DeltaMFacetsRk3}
	Let $M$ be a simple and connected $\kk$-realizable $\mat{3}{n}$-matroid, and $\eta \subset [n]$. Then $\eta$ is nondegenerate if and only if either
	\begin{enumerate}[noitemsep]
		\item $|\eta| = 1$ and $M/\eta$ is connected, or 
		\item $|\eta| = n-1$ and $M|\eta$ is connected, or 
		\item $\eta$ is a line.
	\end{enumerate}
\end{proposition}

\begin{proof}
	First, we claim that if $\eta$ is nondegenerate, then $|\eta| = 1, n-1$ or $\eta$ is a line. To that end, fix a subset $\eta$ such that  $1<|\eta|<n-1$, and $\eta$ is not a line ($\eta$ is clearly degenerate when $|\eta| = 0$ or $n$).  If $|\eta| = 2$, then $M|\eta \cong U(1,1)\oplus U(1,1)$, hence not connected. Otherwise, $2<|\eta|<n-1$ and $\rho_M(\eta)  =2$ or $3$. If $\rho_M(\eta) = 2$, then there is a line $\eta'$ properly containing $\eta$. In this case, every element in $\eta'\setminus \eta$ becomes a loop in $M/\eta$. Because $M/\eta$ has at least 2 elements, having a loop implies that it is not connected.  If $\rho_M(\eta) = 3$, then every element in $[n]\setminus \eta$ is a loop in $M/\eta$. 
	Since $M$ is connected, $n-|\eta|\geq 2$, thus $M/\eta$ is not connected. In all cases, $\eta$ is degenerate, hence the claim.
	
	If $|\eta| = 1$, then $M|\eta \cong U(1,1)$ which is connected,  so $\eta$ is nondegenerate if and only if $M/\eta$ is connected. Similarly, if $|\eta| = n-1$, then $M/\eta \cong U(1,1)$ which is connected, so $\eta$ is nondegenerate if and only if $M|\eta$ is connected.     Finally, suppose $\eta$ is a line. Then $M|\eta \cong U(2,k)$ ($k\geq 3$) and $M/\eta \cong U(1,\ell)$ ($\ell \geq 2$), both of which are connected, so $\eta$ is nondegenerate.   
\end{proof}

\Cref{lem:smoothScheme} and the next two lemmas will allow us to trim down the amount of $M'\lessdot M$  that we will need to check in the proofs of \Cref{prop:smoothmap} and \Cref{prop:37smoothmaps}.

\begin{lemma}
	\label{lem:smooth} 
	Suppose $M$ is simple and connected and $\eta$ is a line of $M$. If $i\in [n]$ is not contained in any line and $\Gr_{M|[n]\setminus i}$ is integral, then $M_{\eta}|([n]\setminus i) \leq M|[n]\setminus i$, and we have a commutative diagram
	\begin{equation*} 
	\begin{tikzcd}
	\arrow[d, "\varphi_{M,M_{\eta}}"']  \Gr_M \arrow[hookrightarrow]{r}{}    & \Gr_{M|[n]\setminus i} \times \GGm^3 \arrow[d, "\varphi_{M|[n]\setminus i,M_{\eta}|[n]\setminus i} \times \pi"] \\
	  \Gr_{M_{\eta}} \arrow[hookrightarrow]{r}{} & \Gr_{M_{\eta}|[n]\setminus i} \times \GGm
	\end{tikzcd}  
	\end{equation*} 
	where the top and bottom arrows are dominant open immersions, and $\pi$ is a coordinate projection. In particular, if $\varphi_{M|[n]\setminus i,M_{\eta}|[n]\setminus i}$ is a SDC-morphism, then so is $\varphi_{M,M_{\eta}}$. 
\end{lemma}

\begin{proof}
	As usual, we use affine coordinates as in \Cref{construction}, assume that $\{0,1,2\}$ is a basis and the first 3 columns of $X$ form the identity matrix. We may also assume that $i=n-1$ and  $0,1\in \eta$.	
	 As in the proof of \Cref{lem:smoothScheme}, the dominant open immersion $\Gr_{M} \hookrightarrow \Gr_{M|[n-1]} \times \GGm^3$ is induced by the inversion of $X_{\beta}$, $\beta\in \cB(M)$ in  $R_{M|[n-1]}^x[x_{0,n-4}^{\pm},x_{1,n-4}^{\pm},x_{2,n-4}^{\pm}]$. Since $M_{\eta} \cong M/\eta \oplus  M|\eta$ and $M/\eta$ has rank 1, all elements of $[n]\setminus \eta$ become parallel to $2$ in $M_{\eta}$, in particular $\lambda_{0,j-3} = \lambda_{1,j-3} = 0$ in $R_{M_{\eta}}^x$ for $j\notin \eta$. Similar to $R_{M}^x$,  $R_{M_{\eta}}^x$ is obtained from $R_{M_{\eta}|[n-1]}^x[x_{2,n-4}^{\pm}]$ by inverting $X_{\beta}$ for $\beta \in \cB(M_{\eta})$. This localization induces a dominant open immersion $\Gr_{M_{\eta}} \hookrightarrow \Gr_{M_{\eta}|[n-1]} \times \GGm$. The morphism  $\GGm^{3} \to \GGm$ is induced by $\kk[x_{2,n-4}^{\pm}] \subset \kk[x_{0,n-4}^{\pm},x_{1,n-4}^{\pm},x_{2,n-4}^{\pm}]$.  Commutativity of the diagram is now a simple verification at the level of rings. The last statement follows from \Cref{prop:SDCcomposition} and the fact that $\pi$ is a SDC-morphism.
\end{proof}

\begin{lemma}
	\label{lem:smooth2}
Suppose $M$ is simple and connected, $\eta$ is a line of $M$, and $\Gr_{M|[n]\setminus i}$ is integral. If $i\in \eta$ is not contained in any other line, then $M_{\eta\setminus i} \leq M|[n]\setminus i$, and we have a commutative diagram
	\begin{equation*} 
	\begin{tikzcd}
	\arrow[d, "\varphi_{M,M_{\eta}}"']  \Gr_M \arrow[hookrightarrow]{r}{}    & \Gr_{M|[n]\setminus i} \times \GGm^2 \arrow[d, "\varphi_{M|[n]\setminus i,M_{\eta\setminus i}} \times \id"] \\
	\Gr_{M_{\eta}} \arrow[hookrightarrow]{r}{} & \Gr_{M_{\eta\setminus i}} \times \GGm^2
	\end{tikzcd}  
	\end{equation*} 
	where the top and bottom arrows are dominant  open immersions.  In particular, if $\varphi_{M|[n]\setminus i,M_{\eta\setminus i}}$ is a SDC-morphism, then so is $\varphi_{M,M_{\eta}}$.  
\end{lemma}

\begin{proof} 
	Similar to the proof of \Cref{lem:smooth}, we use affine coordinates as in \Cref{construction}, assume that $\{0,1,2\}$ is a basis, the first 3 columns of $X$ form the identity matrix, $i=n-1$, and  $0,1,n-1\in \eta$.
	As in the proof of \Cref{lem:smoothScheme}, the dominant open immersion $\Gr_{M} \hookrightarrow \Gr_{M|[n-1]} \times \GGm^2$ is induced by the inversion of $X_{\beta}$, $\beta\in \cB(M)$ in  $R_{M|[n-1]}^x[x_{0,n-4}^{\pm},x_{1,n-4}^{\pm}]$. 
	Since $M_{\eta} \cong M/\eta \oplus M|\eta$ and $M/\eta$ has rank 1, all elements of $[n]\setminus \eta$ become parallel to $2$ in $M_{\eta}$. Because $\{0,2,n-1\}$ and $\{1,2,n-1\}$ remains bases in $M_{\eta}$, $x_{0,n-4}$ and $x_{1,n-4}$ are still invertible in $R_{M_{\eta}}^x$. Similar to $R_{M}^x$,  $R_{M_{\eta}}^x$ is obtained from $R_{M_{\eta\setminus \{n-1\}}}^x[x_{0,n-4}^{\pm},x_{1,n-4}^{\pm}]$ by inverting $X_{\beta}$ for $\beta \in \cB(M_{\eta})$. This localization induces a dominant open immersion $\Gr_{M_{\eta}} \hookrightarrow \Gr_{M_{\eta\setminus \{n-1\}}} \times \GGm^2$. Commutativity of the diagram is now a simple verification at the level of rings. The last statement follows from \Cref{prop:SDCcomposition}.
\end{proof}

\begin{proposition}
	\label{prop:smoothmap}
	Let $M$ be a $\mat{3}{n}$ matroid for $3\leq n \leq 6$ and $M'\leq M$. Then $\varphi_{M,M'}: \Gr_{M} \rightarrow \Gr_{M'}$ is a SDC-morphism. 
\end{proposition}

\begin{proof} 
By Lemmas \ref{lem:reduceToConnected} and \ref{lem:parallelTSC}, it suffices to consider pairs of matroids of the form $M'\lessdot M$ where $M$ is simple and connected. The only $\mat{3}{3}$-matroid is $U(3,3)$, and $\Delta(3,3)$ is a point, so there is nothing to check.  For $n=4$, the only simple and connected matroid is $U(3,4)$ so $\varphi_{M,M'}$ is a SDC-morphism by \Cref{pr:uniform}. The case $n=5$ follows from \Cref{prop:rk2smoothmap} and \Cref{lm:DualsRespectFaces}.

Finally consider $n=6$. We may assume that $M$ and $M^*$ are simple by \Cref{lm:DualsRespectFaces}, and  $M'=M_{\eta}$ where $|\eta|=1,n-1$ or $\eta$ is a line by \Cref{lm:DeltaMFacetsRk3}. 
It suffices to consider pairs $M_{\eta}\lessdot M$ such that one of the following holds:
\begin{enumerate}[noitemsep]
	\item $\eta = [n]\setminus \{i\}$ and $i$ is contained in 3 or more lines (\Cref{lem:smoothScheme}),
	\item $\eta = \{i\}$ and $i$ is contained in 3 or more lines of $M^*$ (\Cref{lm:DualsRespectFaces}), or
	\item $\eta\in \cL(M)$, every $i\in [n]$ is contained in a line, (\Cref{lem:smooth}) and every $j\in \eta$ is contained in another line (\Cref{lem:smooth2}).
\end{enumerate}
 For $\mat{3}{6}$ matroids (1) and (2) can never happen. Up to symmetry, there is only one pair that satisfies (3):
\begin{equation*}
\cL(M) =\{ \{0,1,3\}, \{0,2,4\}, \{1,2,5\}, \{3,4,5\}\}
\end{equation*}
and $\eta = \{0,1,3\}$.   By the isomorphism $M_{\eta} \cong M/\eta \oplus M|\eta$,  $\{0,1,3\}$ is the only line of $M_{\eta}$ and $2,4,5$ are parallel to each other.   We use affine coordinates as in \Cref{construction}. Assume that the first 3 columns of $X$ form the identity matrix, so $R_{M'}^{x} = \kk[x_{00}^{\pm},x_{10}^{\pm},x_{11}^{\pm},x_{22}^{\pm}]$ and 
\begin{equation*} R_{M}^x = R_{M'}^x[ x_{01}^{\pm}, x_{12}^{\pm}, x_{21}^{\pm}] /\left\langle  x_{00}x_{12}x_{21}+x_{01}x_{10}x_{22} \right\rangle  \cong R_{M'}^x[x_{12}^{\pm},x_{21}^{\pm}].
\end{equation*} 
Then $R_{M_{\eta}}^x \rightarrow R_M^x$ may be identified with the inclusion $R_{M_{\eta}}^x \subset R_{M_{\eta}}^x[x_{12}^{\pm},x_{21}^{\pm}]$ and therefore $\varphi_{M,M_{\eta}}$ is a SDC-morphism. 
\end{proof}

\begin{proposition}
	\label{prop:37smoothmaps}
	Let $M$ be a $\mat{3}{7}$-matroid and $M'\leq M$ such that $Q_{M'}$ is not a face of $\Delta(3,7)$. Then $\varphi_{M,M'}: \Gr_{M} \rightarrow \Gr_{M'}$ is a SDC-morphism. 
\end{proposition}

\begin{proof}
	By Lemmas \ref{lem:reduceToConnected}, \ref{lm:simpleReduction}, \ref{lem:smooth} and \ref{lem:smooth2}, we may assume that $M$ is simple, connected, every element in $[7]$ is contained in a line, and there is a  line $\eta$ with the property that every $i\in \eta$ is contained in another line. There are only six such matroids. We list these in~\Cref{table:37}, together with their nondegenerate subsets (up to symmetry) that define internal facets, i.e., those facets that are not faces of $\Delta(3,7)$. This has the effect of excluding the subsets of size $1$ or $6$.  The representatives of the nondegenerate subsets are chosen so that $\{0,1,2\}$ is a basis of both $M_{\eta}$ and $M$ whenever we need to perform an explicit computation in affine coordinates. 
\begin{center}
	\begin{table}[tbh]
		\centering
		\begin{tabular}{|l |l| l|} 
			\hline
			&$\cL(M)$ & Internal Facets \\
			& & ($\Aut(M)$-representatives)  \\ [0.2ex] 
			\hline\hline
			7.1	&$\{0,1,3\}, \{0,2,4\}, \{1,2,5\}, \{2,3,6\},  \{4,5,6\}$ & (1) $\{0,1,3\}$, (2) $\{0,2,4\}$  \\ 
			
			\hline
			7.2	&$\{0,1,3\}, \{0,2,4\}, \{0,5,6\}, \{1,2,5\}, \{1,4,6\},$ & (1) $\{0,1,3\}$ \\
			&$ \{2,3,6\} $   &\\ 
			\hline
			7.3	&$\{0,1,3\}, \{0,2,4\}, \{1,2,5\},  \{4,5,6\} $ & (1) $\{0,1,3\}$, (2) $\{0,2,4\}$ \\ 
			\hline
			7.4	&$\{0,1,3\}, \{0,2,4\}, \{1,2,5\}, \{2,3,6\}$ & (1) $\{0,1,3\}$, (2) $\{0,2,4\}$  \\ 
			\hline
			7.5	&$\{0,1,3\}, \{0,2,4\}, \{1,2,5\}, \{1,4,6\},  \{2,3,6\}  $ & (1) $\{0,1,3\}$, (2) $\{1,2,5\}$  \\ 
			\hline
			7.6	&$\{0,1,5\}, \{0,2,3,6\}, \{1,4,6\}, \{3,4,5\}  $ & (1) $\{0,1,5\}$,   (2) $\{0,2,3,6\}$   \\ 
			\hline
		\end{tabular}
		\caption{The  simple connected rank $3$ matroids on $[7]$ relevant to \Cref{prop:37smoothmaps}, together with nondegenerate subsets defining internal facets}
		\label{table:37}
	\end{table}
\end{center}

Cases 7.3(1), 7.4(2), 7.5(2), 7.6(2) follow from~\Cref{lem:smooth}, and case 7.6(1) is similar to the case worked out in the proof of \Cref{prop:smoothmap} (indeed, the matroid in 7.6 is obtained by adding an element to a line of this matroid). For these remaining cases, we proceed by a direct computation using affine coordinates as in \Cref{construction}. The first 3 columns of $X$ will always be the identity matrix.   As in the proof of \Cref{prop:smoothmap}, the isomorphism $M_{\eta} \cong M/\eta \oplus M|\eta$ gives a simple way to identify $M_{\eta}$.

  Let $M'\lessdot M$ be the pair in Case 7.1(1).  Then $R_{M'}^x =\kk[x_{00}^{\pm},x_{10}^{\pm},x_{21}^{\pm},$ $x_{22}^{\pm}, x_{23}^{\pm}]$ and $R_{M}^x$ is the quotient of $(S_{M}^{x})^{-1}R_{M'}^x [x_{01}^{\pm}, x_{12}^{\pm}, x_{03}^{\pm}, x_{13}^{\pm}]$ by the ideal
\begin{equation*} 
 \left\langle x_{00}x_{13} - x_{10}x_{03},\ 
x_{01}(x_{12}x_{23} - x_{22}x_{13}) - x_{21}x_{12}x_{03}  \right\rangle.
 \end{equation*} 
 Because $\overline{X}_{056} = x_{12}x_{23} - x_{22}x_{13}$ is in $S_{M}^{x}$, we have $R_{M}^x \cong R_{M'}^x[x_{12}^{\pm},x_{13}^{\pm}]$. So $R_{M'}^x \rightarrow R_M^x$ may be identified with the inclusion $R_{M'}^x \subset (S_M^x)^{-1}R_{M'}^x[x_{12}^{\pm}, x_{13}^{\pm}]$.  Therefore  $\varphi_{M,M'}$ is a SDC-morphism.

Next consider the pair $M'\lessdot M$ in case 7.1(2).  Then $R_{M'}^x = \kk[x_{10}^{\pm},x_{01}^{\pm},x_{21}^{\pm}, $ $x_{12}^{\pm},x_{13}^{\pm}]$. By eliminating the variables $x_{00}$ and $x_{22}$ from $R_M^x$, we may identify the morphism $R_{M'}^x\rightarrow R_M^x$ with the inclusion $R_{M'}^x \subset (S_{M}^x)^{-1}R_{M'}^x[x_{03}^{\pm}, x_{23}^{\pm}]$.  
Therefore $\varphi_{M,M'}$ is SDC.  Because the matroid in 7.3 is obtained from $M$ by removing one line, case 7.3(2) is similar. 

Finally consider the pair $M'\lessdot M$ in case 7.2(1).  Then $R_{M'}^x = \kk[x_{00}^{\pm},x_{10}^{\pm},x_{21}^{\pm},$ $x_{22}^{\pm},$ $x_{23}^{\pm}]$ and $R_{M}^x$ is the quotient of $(S_M^{x})^{-1}R_{M'}^x[x_{01}^{\pm},x_{12}^{\pm},x_{03}^{\pm},x_{13}^{\pm}]$ by the ideal
\begin{equation*}
 \langle x_{12}x_{23} - x_{13}x_{22}, x_{01}x_{23} - x_{21}x_{03}, x_{00}x_{13} - x_{10}x_{03}  \rangle.
\end{equation*}
By eliminating the variables $x_{13}$, $x_{03}$, and $x_{12}$, 
 $R_{M'}^x \rightarrow R_M^x$ may be identified with the inclusion $R_{M'}^x \subset (S_M^{x})^{-1}R_{M'}^x[x_{01}^{\pm}]$. Therefore $\varphi_{M,M'}$ is SDC.  Because the matroids in 7.4 and 7.5 are obtained by removing two, resp. one, lines from $M$, cases 7.4(1) and 7.5(1) are similar. 
\end{proof}

\section{Smoothness and irreducibility of initial degenerations}
\label{sec:smoothIrreducibleDeg}

Let $M$ be a $\kk$-realizable $\mat{2}{n}$, $\mat{3}{6}$, or $\mat{3}{7}$ matroid. We compile the results of the previous sections to prove the following more general version of \Cref{theoremA}.

\begin{theorem}
	\label{thm:initDegTSC}
The initial degenerations  $\inw_w\Gr_{M}$ are smooth and irreducible for all $w\in \TGr_M$.
\end{theorem}

\noindent By \Cref{cor:indegMSDIso}, we must show that $\Gr_{M,w}$  is smooth, irreducible, and has the same dimension as $\Gr_M$. Thanks to \Cref{pr:FromAllToCodim1Diagram}, we may compute  $\Gr_{M,w}$ as a limit over a diagram induced by a graph as in \Cref{ex:diagramDualGraph}. When $\grdiv{M}{w}$ is a tree, \Cref{lm:limitTree} tells us that $\Gr_{M,w}$ is smooth and irreducible when $\Gr_{M_Q}$ is smooth and irreducible  and $\varphi_{M_Q,M_{Q'}}:\Gr_{M_Q}\to \Gr_{M_{Q'}}$ is a SDC-morphism for  $Q\in \TC{M,w}$ and $Q'\lessdot Q$ not a face of the hypersimplex. This is illustrated in \Cref{ex:limitTree}. 
However, when $\Gamma_{M,w}$ is not a tree, this data is insufficient to conclude that $\Gr_{M,w}$ is smooth and irreducible,  see  \Cref{rmk:equalizerCounterexample}. 

Let $\eta\subset [n]$. In the examples below and in \Cref{sec:higherGrassmannians}, we encounter the matroids $M(\eta)$ and $M(\eta)'$ defined by
\begin{equation}
\label{eq:Mlambda}
\!\cB(M(\eta)) =  \left\{\left. \beta \in \textstyle{[n]\choose 3}   \right|  |\beta\cap \eta | \geq 2 \right\}, \hspace{1pt} \cB(M(\eta)') =  \left\{\left. \beta \in \textstyle{[n]\choose 3}  \right|   |\beta\cap \eta| = 2 \right\}.
\end{equation}
A simple computation in affine coordinates yields $\dim \Gr_{M(\eta)}  = n+2|\lambda|-7$, and $\dim \Gr_{M(\eta)'} = n+|\lambda|-5$. Also, we set  $f_{\lambda} = e_{\lambda_1} \wedge \cdots \wedge e_{\lambda_d} \in \wedge^d\RR^n$ for  $\lambda = \{\lambda_1,\ldots,\lambda_d\}$.

\begin{example}
	\label{ex:limitTree}
Let
\begin{equation*}
w=f_{013}+f_{024}+f_{056}+f_{125}+f_{146}+f_{236}
\end{equation*}
and $C$ matroid 7.2  from \Cref{table:37}.    The adjacency graph $\Gamma_w$ is a star tree with $v_{Q_C}$ as the central vertex and a leaf vertex $v_{Q_{M(ijk)}}$  for each $\{i,j,k\}\in \cL(C)$.  The edge between $C$ and $M(ijk)$ corresponds to the matroid $M(ijk)'$. Because $\Delta_w$ is a matroid subdivision and does not lie in the relative interior of the Fano cone as in \cite[Theorem~2.1]{HJJS09}, $w\in \TGrU{3}{7}$. The isomorphism from \Cref{pr:FromAllToCodim1Diagram} yields
\begin{equation*}
 \Gr_{w} \cong \Gr_C \times_{\prod \Gr_{M(ijk)'}} \prod \Gr_{M(ijk)}.
\end{equation*}
 The thin Schubert cell $\Gr_C$ is smooth and irreducible by \Cref{prop:37SmoothIrreducible} and the  $\Gr_{M(ijk)} \rightarrow \Gr_{M(ijk)'}$ are SDC-morphisms by  \Cref{prop:37smoothmaps}. From the preceding comments, $\dim \Gr_{M(ijk)} = 6$ and $\dim \Gr_{M(ijk)'} = 5$.  A simple computation in affine coordinates yields $\dim \Gr_{C} = 6$. By \Cref{lm:limitTree}, $\Gr_w$ is smooth and irreducible of dimension $12$, as is $\inw_w\GrU{3}{7}$ by \Cref{cor:indegMSDIso}.
\end{example}

When the maximal cells $Q$ of $\Delta_{M,w}$ all share a common vertex, we may determine whether $R_{M,w}^x$ defines a smooth and irreducible $\kk$-scheme by hand, as illustrated in \Cref{ex:affineCoordinates}. However,  many matroid subdivisions  do not have this property, e.g., the subdivision $\Delta(3,7)$ in the previous example.

\begin{example}
	\label{ex:affineCoordinates}
	Let $M$ be the following matroid
	\begin{equation*}
    \cL(M) = \{ \{0,2,4\},\ \{0,3,6\},\ \{1,2,3\},\   \{1,4,6\},\  \{2,5,6\}     \}.
	\end{equation*}
	and set
	\begin{equation*}
	w=- f_{013} + f_{345} - f_{016} + f_{245} + f_{246} + f_{234} + f_{145} - f_{135} - f_{136} + f_{124} + f_{456}.  
	\end{equation*}
	The subdivision $\Delta_{M,w}$ is matroidal.  The matroids of maximal cells and   $\Gamma_{M,w}$ are illustrated in \Cref{fig:affineCoordinatesEx}. Similar to \Cref{ex:limitTree}, we see that $w\in \TGr_M$. Because $\{0,1,2\}$ is a basis for each $M_i$, $R_{M,w}^x$ may be computed using affine coordinates as in \Cref{construction}.  Assume that the first 3 columns of the matrix $X$ form the identity. We have $I_{M_2}^x = I_{M_3}^x = \langle 0 \rangle$, and 
	\begin{align*}
	I_{M_0}^x = \langle x_{10}x_{23} - x_{13}x_{20}, x_{01}x_{23} - x_{21}x_{03} \rangle,  && I_{M_1}^x = \langle x_{02}x_{13} - x_{03}x_{12} \rangle. 
	\end{align*}
	Therefore $I_{M,w}^x = \langle x_{10}x_{23} - x_{13}x_{20}, x_{02}x_{13} - x_{03}x_{12}, x_{01}x_{23} - x_{21}x_{03}  \rangle$. Because $x_{03},x_{13}$ are in $S_{M,w}^x$, we may solve for these variables to produce an isomorphism
	\begin{equation*}
	R_{M,w}^x \cong (S_{M,w}^x)^{-1} \kk[x_{01}^{\pm},x_{02}^{\pm},x_{10}^{\pm},x_{12}^{\pm}, x_{20}^{\pm},x_{21}^{\pm},x_{22}^{\pm},x_{23}^{\pm}].
	\end{equation*}
	This realizes $\Gr_{M,w}$ as an open subscheme of $\GGm^8$. Therefore $\Gr_{M,w}$ is smooth and irreducible of dimension $8$, as is $\inw_w\Gr_{M}$ by \Cref{cor:indegMSDIso}.
\end{example}

\begin{figure}[tbh]
	\centering
	\includegraphics[scale=0.6]{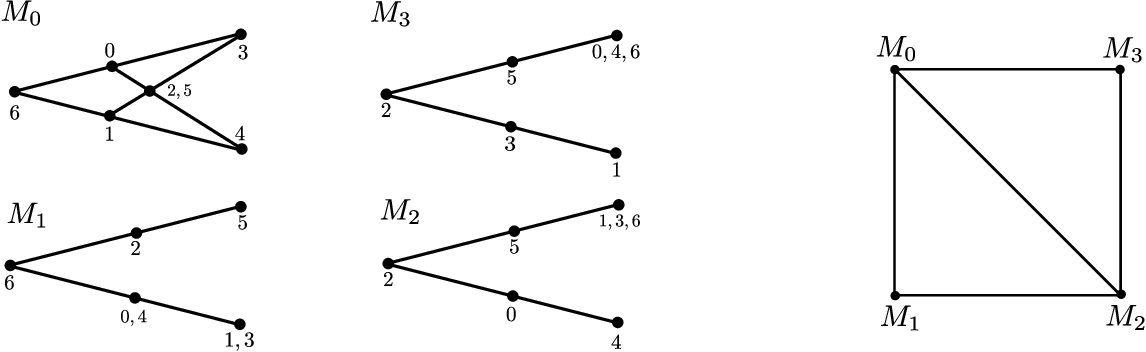}
	\caption{The matroids and adjacency graph appearing in \Cref{ex:affineCoordinates}} \label{fig:affineCoordinatesEx}
\end{figure}

In general, we use a combination of the above techniques to show that all of the $\tsc{M,w}$ are smooth and irreducible. \Cref{lm:centerlimit} handles the case where $\grdiv{M}{w}$ has no leaves, showing that $\Gr_{M,w}$ is smooth and irreducible by a direct analysis of $R_{M,w}^x$ (for these subdivisions, all  maximal cells share a common vertex).  We take care of the remaining cases using this lemma together with \Cref{prop:connectedFiberProduct}, which considers the behavior of smoothness and irreducibility under pullbacks.

\begin{lemma}\label{lm:centerlimit}
	Let $M$ be a $\kk$-realizable, rank $3$-matroid on $[6]$ or $[7]$, and $w\in \TGr_{M}$ such that $\grdiv{M}{w}$ has no leaves. Then $\tsc{M,w}$ is smooth an irreducible. 
\end{lemma}

\begin{proof}
	By \Cref{lm:simpleIsEnoughInitialDeg}, we may assume that $M$ is simple.  We will work with affine coordinates as in \Cref{construction}, and follow a strategy similar to \Cref{ex:affineCoordinates}.  First, suppose $e_{\beta}$ common to all $Q\in \TC{M,w}$ (if such a vertex exists). Let the columns of $X$ prescribed by $\beta$ be the identity matrix.
	  Let 
	  \begin{equation*}
	  g_{M_j}^{x} = \prod_{\lambda\in\cB(M_j)} \overline{X}_{\lambda} \hspace{10pt} \text{ and }  \hspace{10pt} g_{M,w}^{x} = \prod_{Q_{M_j} \text{ maximal}}g_{M_j}^{x}.
	  \end{equation*}
	  Finally, let $J_{M,w}^x =\left( I_{M,w}^x : (g_{M,w}^x)^{\infty} \right) \subset B_M$ (saturation here has the effect of removing the primary components of the irrelevant ideal).  
    By \Cref{prop:limitideal}, to show that $\Gr_{M,w}$ is smooth an irreducible, it suffices to show that  the the quotient of $\kk[x_{ij}^{\pm} \, | \, \lambda_{ij}\in \cB(M)]$ by the extension of $J_{M,w}^x$ 
	is isomorphic to a Laurent polynomial ring.   Due to the large number of cases that we need to check, we will find it more convenient to show that $\kk[x_{ij}^{\pm}]/(J_{M,w}^x\cdot \kk[x_{ij}^{\pm}])$ has this property. 
	
	We  proceed by a direct computation, using computer assistance. (We  emphasize that  this computation may be carried out by hand for any individual $w$, once $\Delta_{M,w}$ is computed. We use a computer due to the large number of cases.) Representatives  $w$ of the cones in $\cG_{M}$ were computed in proof of~\Cref{prop:3nDimension}, along with the  subdivisions $\Delta_{M,w}$. We use ~\MT~ to compute the adjacency graphs and catalog those $w$ such that $\grdiv{M}{w}$ has no leaves. There are 17 such graphs among all simple $\mat{3}{6}$-matroids, and $877$ for  $\mat{3}{7}$.  For each such $(M,w)$, there is a vertex $e_{\beta}$ common all of the  $Q\in \TC{M,w}$.  We choose such a  $\beta$ that is maximal with respect to the \texttt{revLex} order, compute $J_{M,w}^x$ as above, and consider its extension to $\kk[x_{ij}^{\pm}]$.  While this produces a large number of ideals, many end up being the same. For $\mat{3}{6}$, computing the ideals takes about 15 seconds, and there are 3 unique ideals:
	 \begin{equation*}
	 \langle 0\rangle, \langle x_{02}x_{11} - x_{01}x_{12} \rangle, \langle x_{02}x_{10} - x_{00}x_{12} \rangle \subset \kk[x_{ij}^{\pm}].
	 \end{equation*}
	 By solving for $x_{12}$ in the last two ideals, we see that the quotients $\kk[x_{ij}^{\pm}]/(J_{M,w}^x\cdot \kk[x_{ij}^{\pm}])$ are all isomorphic to Laurent polynomial rings.	 For $\mat{3}{7}$, this computation takes about $50$  minutes. 	 We list these ideals in \Cref{appendix:37Data}, together with variables that may be eliminated to produce an isomorphism of $\kk[x_{ij}^{\pm}]/(J_{M,w}^x\cdot \kk[x_{ij}^{\pm}])$ with a Laurent polynomial ring. 
\end{proof}

Let $G$ be a connected graph. Given a leaf-vertex $v$, the \textit{branch} of $G$ containing $v$ is the largest full subgraph of $G$ that contains $v$ and does not meet any cycle of $G$ (note that this is non-standard terminology).   

\begin{theorem}
	\label{thm:limitSmoothIrreducible}
	Let $w\in \TGr_M$ where $M$ is a $\kk$-realizable $\mat{2}{n}$, $\mat{3}{6}$, or $\mat{3}{7}$ matroid. Then $\Gr_{M,w}$ is smooth and irreducible.  
\end{theorem}

\begin{proof}
	By \Cref{pr:FromAllToCodim1Diagram}, we know that $\tsc{M,w}$ is isomorphic to a limit over the adjacency graph $\grdiv{M}{w}$. Since all relevant thin Schubert cells are smooth and irreducible, and all relevant morphisms are SDC, we may use \Cref{lm:limitTree} to conclude that $\tsc{M,w}$ is smooth an irreducible when $\grdiv{M}{w}$ is a tree. In particular, this completes the proof in the $d=2$ case. 
	
	Now suppose $(d,n)=(3,6)$ or $(3,7)$. We need only consider those $w$ such that $\grdiv{M}{w}$ is not a tree.   We proceed by induction on the largest diameter of a branch of $\Gamma_{M,w}$.  When $\grdiv{M}{w}$ has no leaves, $\tsc{M,w}$ is smooth and irreducible from \Cref{lm:centerlimit}, hence the base case of the induction.  	
	
	Let $v_{Q_{1}},\ldots,v_{Q_{k}}$ denote leaf vertices of $\grdiv{M}{w}$, $L_i = M_{Q_{i}}$, and $L_{i}'$ for the matroid corresponding to the edge adjacent to $v_{Q_i}$.  There is a hyperplane $H_k$ in $\RR^n$ such that $\Delta_{M} \cap H_k = \Delta_{L_k'}$, with $\Delta_{L_k}$ in one of the halfspaces of this hyperplane. The polytope given by the intersection of $\Delta_{M}$ with the other halfspace is also a matroid polytope: it is the convex hull of the vertices $e_{\beta}$ such that $\beta\in (\cB(M) \setminus \cB(L_k) ) \cup \cB(L_k')$. The adjacency graph to this subdivision is obtained by removing the vertex and edge corresponding to $\Delta_{L_k}$ and $\Delta_{L_k'}$ respectively. Repeating this procedure for the remaining $L_i$'s, we see that the union of the polytopes corresponding to non-leaf vertices in $\grdiv{M}{w}$ is a matroid polytope. We denote the corresponding matroid by $C$.   	 By   \Cref{pr:limitContractGraph} and \ref{pr:FromAllToCodim1Diagram}, 
	\begin{equation*} 
	\Gr_{M,w} \simeq  \Gr_{C,w} \times_{\prod \tsc{L_i'}}\prod \tsc{L_i}. 
	\end{equation*} 
	The $\kk$-scheme $\Gr_{C,w}$ is smooth and irreducible by the inductive hypothesis and  $\varphi_{L_i,L_i'}$ are SDC-morphisms by~\Cref{prop:smoothmap} and~\Cref{prop:37smoothmaps}. We conclude that $\tsc{M,w}$ is smooth and irreducible by~\Cref{prop:connectedFiberProduct}. 
\end{proof}

\begin{corollary}
	\label{cor:limitRank3}
	If $M$ is a $\kk$-realizable  $\mat{3}{6}$ or $\mat{3}{7}$ matroid and $w\in \TGr_M$, then $\psi_{M,w}:\inw_w\Gr_M \to \Gr_{M,w}$ is an isomorphism. 
\end{corollary}

\begin{proof}
	 By  \Cref{thm:closedEmbedding} and \Cref{prop:3nDimension}, $\psi_{M,w}$ is a closed immersion of affine schemes of the same dimension. Moreover  $\Gr_{M,w}$ is integral by \Cref{thm:limitSmoothIrreducible}. Therefore $\psi_{M,w}$ is an isomorphism by \Cref{prop:closedImmersionIso}.
\end{proof}

\begin{proof}[Proofs of \Cref{theoremA} and \Cref{thm:initDegTSC}]
These theorems follow from  \Cref{prop:limitRank2},  \Cref{thm:limitSmoothIrreducible}, and \Cref{cor:limitRank3}.
\end{proof}

\section{The log canonical compactification of $\XU{3}{7}$}
\label{sec:logCanonical}

\noindent We now prove  \Cref{theoremB}, that the normalization of the Chow quotient of $\GrC{3}{7}$ by the maximal torus $\PGL(7)$ is the log canonical compactification of $\XU{3}{7}$. For background on log minimality and log canonical compactifications, we see the introduction of \cite{HackingKeelTevelev2009}, for  the Chow quotient of $\GrC{d}{n}$, see \cite{Kapranov1993} \cite[Section~2]{KeelTevelev2006}, and for sch\"on compactifications, see \cite{hacking, LuxtonQu, Tevelev2007}. 
Throughout this section, we use the following notation for polyhedral fans and toric varieties  that is consistent with \cite{Fulton}.   Let $N$ be a lattice, $T_{N}$ its torus, and $\Sigma$ a rational polyhedral fan in $N_{\RR}$. When $T$ is a torus, we write $N_{T}$ for its cocharacter lattice.  Given a cone $\sigma$ of $\Sigma$, $N_{\sigma}$ denotes the saturated sublattice of $N$ generated by $\sigma\cap N$, $N(\sigma) = N/N_{\sigma}$, and $\St(\sigma)$ the star of $\sigma$, viewed as a fan in $N(\sigma)_{\RR}$.
We write $X(\Sigma)$ for the toric variety of $\Sigma$.

Let $H \subset \PGL(n)$ be the maximal torus, $M$ a loop-free matroid, and $T(M)$ the dense torus of $\Proj(B_M)$. As before, we let $\{e_i \, | \, i\in [n]\}$ denote the standard basis of $\ZZ^n$ and $f_{\lambda} = e_{\lambda_1}\wedge \cdots \wedge e_{\lambda_d}$ for $\lambda = \{\lambda_1,\ldots,\lambda_d \}$. 
The cocharacter lattices of $H$ and $T(M)$ isomorphic to   $\ZZ^n/\zspan{\mathbf{1}}$ and $\ZZ^{\cB(M)}/\zspan{\mathbf{1}}$, respectively.  The torus $H$ embeds into $T(M)$ by 
\begin{equation}
\label{eq:HinT}
N_H \to N_{T(M)} \;\;\;\;\;\;\;\;
e_{i} \mapsto \sum_{\lambda\ni i} f_{\lambda}. 
\end{equation}
Thus $H$ acts on $\Proj(B_M)$ via the action of $T(M)$. This restricts to a free action on $\Gr_M$, and we set $X_M = \Gr_M/H$. The quotient $\Gr_M \to X_M$ is induced by a monomial ring map \cite[Proposition~2.1]{GibneyMaclagan}, therefore    $\Trop X_M = \TGr_M/(N_H)_{\RR}$.

Now we focus on $M=U(d,n)$;  In this case we write $T=T(U(d,n))$ and  $\XU{d}{n}= X_{U(d,n)}$. The Pl\"ucker embedding induces an closed immersion  of Chow quotients 
$\GrC{d}{n}/\!/H \hookrightarrow \PP(\wedge^d\kk^n) /\!/H$.  By \cite{Kapranov1993, KapranovSturmfelsZelevinsky}, the normalization of $\PP(\wedge^d\kk^n) /\!/H$ is the toric variety $Y_{d,n} := X(\SFU{d}{n}/(N_H)_{\RR})$.   
Let $\XCL{d}{n}$ be the closure of $\XU{d}{n}$ in $Y_{d,n}$.  Then  $Y_{d,n} \to   \PP(\wedge^d\kk^n) /\!/H$ induces a birational morphism  $\XCL{d}{n} \to \GrC{d}{n}/\!/H$, thus both have the same normalization, which we denote by $\XC{d}{n}$. When $n=6,7$,  $\XC{3}{n}$ is also the closure of $\XU{3}{n}$ in $X(\mathcal{S}_{3,n}')$ where $\mathcal{S}_{3,n}' = \mathcal{S}_{3,n}/(N_{H})_{\RR}$.

\begin{lemma}
	\label{lem:X37schon}
	The initial degenerations of $\XU{3}{7}$ are smooth and irreducible. In particular, $\XC{3}{7}$ is a sch\"on compactification of $\XU{3}{7}$.
\end{lemma}

\begin{proof}
Let $N_{H}^{\sat}$ denote the saturation of the image of the map from \Cref{eq:HinT}. A splitting of the exact sequence $0\to N_{H}^{\sat} \to N_T \to N_{T/H} \to 0$  induces an isomorphism  $\GrU{d}{n} \cong \XU{d}{n} \times H$.  As stated earlier, this is monomial at the level of coordinate rings.
	 Therefore $\inw_w\GrU{3}{7} \cong \inw_{\tilde{w}} \XU{3}{7} \times H$ where $\tilde{w}$ is the projection of $w$ to $(N_{T/H})_{\RR}$. The first statement now follows from  \Cref{theoremA}. 
	By~\cite[Theorem~1.5]{LuxtonQu} $\XCL{3}{7}$ is a sch\"on compactification, which is already normal by \cite[Theorem~1.4]{Tevelev2007}.
\end{proof}

\begin{proof}[Proof of  \Cref{theoremB}]
By \Cref{lem:X37schon}, $X$ is a sch\"on compactification of $\XU{3}{7}$. Let $B$ the boundary divisor of $\XU{3}{7}\subset \XC{3}{7}$. To show that $K_{\XC{3}{7}}+B$ is ample, we follow a strategy laid out in \cite{Luxton2008} based on \cite{HackingKeelTevelev2009}.
	
For each cone $\sigma \in \LFU{3}{7}'$, let $X_{\sigma}$ denote the locally closed stratum of $\XC{3}{7}$ in the corresponding torus orbit of $X(\LFU{3}{7}')$. There is an isomorphism  $\inw_w\XU{3}{7} \cong X_{\sigma} \times T_{N_{\sigma}}$ for any $w$ in the relative interior of  $\sigma$ \cite[Lemma~3.6]{HelmKatz}, so each $X_{\sigma}$ is smooth and irreducible by \Cref{lem:X37schon}. 	Because $\XC{3}{7}$ is a sch\"on compactification,  $(\XC{3}{7},B)$ has at worst toroidal singularities \cite[Theorem~1.4]{Tevelev2007}. By \cite[Theorem~9.1]{HackingKeelTevelev2009}, $K_{\XC{3}{7}}+B$ is ample if and only if each $X_{\sigma}$ is log minimal. We know that $\XU{3}{7}$ is log minimal by \cite[Proposition~2.18]{KeelTevelev2006}, so we need only consider the $X_{\sigma}$ for $\sigma \neq 0$. 
	
By \cite[Lemma~3.3.6]{MaclaganSturmfels2015} $\Trop X_{\sigma}$ is the underlying set of $\St(\sigma)$ in $N(\sigma)_{\RR}$. 	  The stratum $X_{\sigma}$ is sch\"on because its closure in $X(\St(\sigma))$ is a sch\"on compactification. Therefore, $X_{\sigma}$  is either log minimal or preserved by a nontrivial subtorus $S\subset T_{N(\sigma)}$ \cite[Theorem~3.1]{HackingKeelTevelev2009}, which occurs if and only if $\Trop X_{\sigma}$ is invariant under  translation by the subspace $(N_S)_{\RR} \subset N(\sigma)_{\RR}$  \cite[Lemma~5.2]{KatzPayne2011}. So it suffices to show that each $\Trop X_{\sigma}$ is not invariant under translation by any rational subspace of $N(\sigma)_{\RR}$. 	We prove this in \Cref{lem:raysOrigin}, using the necessary condition for such subspaces in  \Cref{lem:translation}.
\end{proof}

\begin{lemma}
	\label{lem:translation}
	Suppose $\Sigma$ is a rational polyhedral fan in $N_{\RR}$ that is invariant under translation by the linear subspace $V\subset N_{\RR}$. Then $V\subset (N_{\sigma})_{\RR}$ for every maximal cone $\sigma$ of $\Sigma$.
\end{lemma}

\begin{proof}
	If $\sigma$ is a maximal cone such that $V\not\subset N_{\sigma}$ then $\dim(V+N_{\sigma})>\dim \Sigma$, therefore $V$ cannot preserve $\Sigma$. 
\end{proof}

\begin{lemma}
	\label{lem:raysOrigin}
For each cone  $\sigma$ of $\LFU{3}{7}'$, $\Trop(X_{\sigma})$ is not preserved under translation by any rational subspace of $N(\sigma)_{\RR}$.
\end{lemma}

\begin{proof}
 The case $\sigma=0$ follows from the fact that $\XU{3}{7}$ is sch\"on and log minimal as in the proof of \Cref{theoremB}, so we focus on $\sigma\neq 0$.  By \Cref{lem:translation},  to prove that $\Trop(X_{\sigma})$ is not preserved under translation by any rational subspace of $N(\sigma)_{\RR}$ it suffices to show
	\begin{equation}
	\label{eq:intersectionStar}
	(N_{\sigma})_{\RR} = \bigcap\, (N_{\tau})_{\RR}
	\end{equation}
	where the intersection is taken over all maximal cones $\tau\in\St(\sigma)$. By symmetry, it suffices to show \Cref{eq:intersectionStar} for a collection of $S_7$-orbit representatives of the $\sigma$.
	
The Gr\"obner fan $\cG_{3,7}$ was computed in~\gfan~(as before), and we use~\sage~to compute $\LFU{3}{7}'$ by grouping together those cones that correspond to the same matroid subdivision of $\Delta(3,7)$.  	The $f$-vector (starting at dimension 0) for $\LFU{3}{7}'$ up to $S_7$-symmetry is
		\begin{equation*}
		f(\LFU{3}{7}' \mod S_7) = (1, 5, 30, 107, 217, 218, 94). 
		\end{equation*}
	For each representative we compute $\St(\sigma)$ and the intersection in \Cref{eq:intersectionStar}, also in~\sage. This part of the computation may be completed in under 5 minutes on a standard desktop computer.
\end{proof}

\begin{remark}
	\label{rmk:Luxton}
	A direct adaptation of Luxton's methods for proving that $\XU{3}{6}$ is sch\"on does not work for $\XU{3}{7}$, as we now describe.
	For a degree $9-n$ del Pezzo surface $S$, let $e_1,\ldots, e_n, h$ denote the standard generators of $\Pic S$, and $K = 3h-\sum e_i$ the canonical class. In this remark, we focus on the cases $n=6,7$.  The subspace  $K^{\perp}$ contains the root system $E_n$, let $\Lambda$ be the $\ZZ$-lattice generated by  $E_n$.  Set $\alpha_{ij} = e_i-e_j$,  $\beta = 2h - \sum e_i$  (when $n=6$), and $\beta_j = 2h - \sum_{i\neq j} e_i$ (when $n=7$). 
	
	As in the introduction,  $Y^n$ denotes the moduli space of smooth marked del Pezzo surfaces of degree $9-n$, and $\cF_n$ its log canonical fan, whose support is $\Trop Y^n$. We recall the description of $\cF_n$ in  \cite{HackingKeelTevelev2009}. 
	There is an exact sequences of free abelian groups
	\begin{equation*}
	0 \to \Sym^2\Lambda^{\vee} \xrightarrow{\phi} \ZZ^{(E_n)_{+}} \xrightarrow{\psi} N(E_n) \to 0
	\end{equation*}  
	where $\phi(f) = \sum_{\alpha\in (E_n)_{+}} f(\alpha)\cdot \alpha$.  Given a root subsystem $\Theta$, let $\psi(\Theta) = \sum \psi(\alpha)$ where the sum is over the positive roots of $\Theta$. The set of  $A_i$ root subsystems of $E_n$ is denoted by $\cA_i$ and let
	\begin{equation*}
	\cR(E_6) = \cA_1 \sqcup (\cA_2\times \cA_2 \times \cA_2), \hspace{30pt} \cR(E_7) = \cA_1\sqcup \cA_2 \sqcup (\cA_3\times\cA_3) \sqcup \cA_7.
	\end{equation*}
	The rays of $\cF_n$ are $\psi(\Theta)$ for any $\Theta\in \cR(E_n)$, and $\psi(\Theta_1),\ldots, \psi(\Theta_{k})$ span a cone if and only if $\Theta_{i}\perp \Theta_j$,  $\Theta_i\subset \Theta_j$, or $\Theta_j\subset \Theta_i$ (when $n=7$, exclude the ``Fano'' simplices spanned by 7 mutually orthogonal $A_1$ subsystems).

	The open immersion  $Y^{n} \hookrightarrow \XU{3}{n}$ induces a surjective map $\pi:\Trop Y^n \to \Trop \XU{3}{n}$ \cite[Proposition~3.1]{Tevelev2007}. When $n=6$, $\pi$ is induced by a quotient of $N(E_6)_{\RR}$ by  $\Span_{\RR}\{\psi(\beta)\}$.	One step to prove that $\XU{3}{6}$ is sch\"on in Luxton's thesis was to show that the morphism of toric varieties $X(\cF_6) \to X(\LFU{3}{6}')$ is smooth, see \cite[Theorem~4.2.2]{Luxton2008}.  
    When $n=7$, the map $\pi:\Trop Y^n \to \Trop \XU{3}{n}$ is induced by a quotient  of $N(E_7)_{\RR}$ by $\Span_{\RR}\{\psi(\beta_1),\ldots,\psi(\beta_7)\}$.  We claim that the morphism $X(\cF_7) \to X(\LFU{3}{7}')$ is not smooth. 	
	 Let $\Theta_1 = \{\pm \alpha_{ij}\}$ (an $A_1$-subsystem) and  $\Theta_2 = \{\pm \alpha_{ij}, \pm \beta_i, \pm \beta_j \}$  (an $A_2$-subsystem). Because $\Theta_1\subset \Theta_2$, $\sigma = \Span_{\RR_{\geq 0}} \{ \psi(\Theta_1), \psi(\Theta_2)\}$ is a cone in $\cF_7$.  Then $\pi$ maps both rays $\psi(\Theta_1)$, $\psi(\Theta_2)$ to the same ray $\tau$ in $\Trop \XU{3}{7}$. Therefore the restriction of $X(\cF_7) \to X(\LFU{3}{7}')$ to the toric open sets $U_{\sigma} \to U_{\tau}$ is not smooth. 
\end{remark}

\section{Behavior for higher Grassmannians}
\label{sec:higherGrassmannians}

The algebraic properties of both the initial degenerations and the maps between thin Schubert cells that played a central role in the proof of~\Cref{theoremA} fail to hold outside $(d,n)=(2,n)$, $(3,6)$, 
$(3,7)$ and their duals. In this section, we give examples for $d=3$ and $n=8,9$ that show our proof-techniques will not apply beyond the cases treated earlier.
We begin by showing how the analog of \Cref{prop:3nDimension} does not hold when $\charF \kk = 2$ or $n\geq 9$. 

\begin{figure}[tbh]
	\centering
	\includegraphics[scale=0.6]{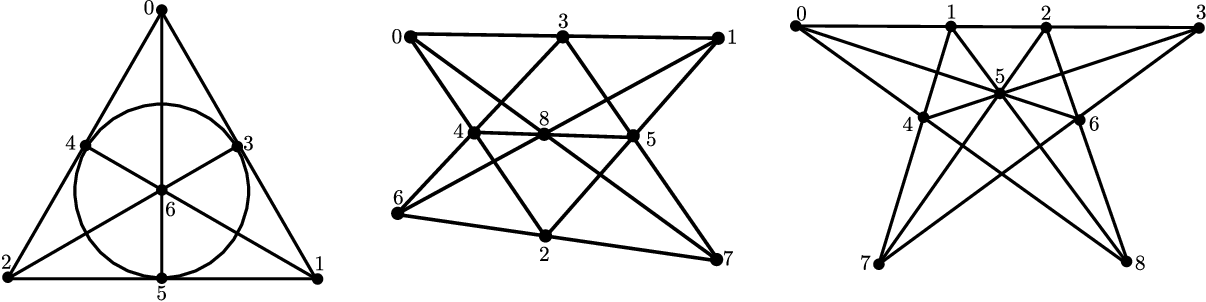}
	\caption{From left to right: the Fano matroid, the Pappus matroid, and the Perles matroid} \label{fig:higherGrassmannians}
\end{figure}

\begin{example}
	\label{ex:37counterexample}
	For this example, suppose $\kk=\overline{\FF}_{2}$. The analog of  \Cref{prop:3nDimension} does not hold in this setting.  Let $F$ be the Fano matroid, i.e. the matroid whose set of lines is 
	\begin{equation*}
	\cL(F) = \{\{0,1,3\},\, \{0,2,4\},\, \{0,5,6\},\, \{1,2,5\},\, \{1,4,6\},\,  \{2,3,6\},\, \{3,4,5\}\}
	\end{equation*}
	as illustrated in  \Cref{fig:higherGrassmannians}. Let $w_F \in \wedge^3\ZZ^{7}$ be the vector
	\begin{equation*}
	w_F =  f_{013} + f_{024} + f_{056} + f_{125} + f_{146} + f_{236} + f_{345}.
	\end{equation*}
	This point lies in  $\TGrU{3}{7}$, as it is the coordinatewise $t$-adic valuation of the Pl\"ucker coordinates of the $\FF_4(\!(t)\!)$-valued matrix 
	\begin{equation*}
	\begin{pmatrix}
	1 & 0 & 0 & 1 & 1 & t & 1 \\
	0 & 1 & 0 & 1 & t & 1 & 1+t \\
	0 & 0 & 1 & t & 1 & 1 & 1+at
	\end{pmatrix}
	\end{equation*}
	where $a\in \FF_4\setminus \{0,1\}$.  The adjacency graph $\Gamma_w$ is a star tree whose central node is $v_{Q_F}$ and has a leaf vertex $v_{Q_{M(ijk)}}$ for each $\{i,j,k\}\in \cL(F)$. The edge adjacent to  $v_{Q_{M(ijk)}}$ corresponds to $M(ijk)'$. A computation in affine coordinates yields $\dim \Gr_F = 6$, hence $\dim \Gr_w = 13$ by   \Cref{prop:closedImmersionIso}. 
\end{example}

\begin{example}
	\label{ex:39counterexample}
	The analog of  \Cref{prop:3nDimension} is also not true for other $(d,n)$, even in characteristic 0. Consider the case $(d,n) = (3,9)$ and let $M_{\Pa}$ be the Pappus matroid, i.e. that matroid set of lines is 
	\begin{align*}
	\cL(M_{\Pa}) = \{&\{0,1,3\},\, \{0,2,4\},\, \{0,7,8\},\, \{1,2,5\},\, \\
	 &\{1,6,8\},  \{2,6,7\},\, \{3,4,6\},\, \{3,5,7\},\, \{4,5,8\} \}
	\end{align*}
	as illustrated in  \Cref{fig:higherGrassmannians}.	Let  $w_{\Pa} \in \wedge^3\ZZ^{9}$ be the vector defined by
	\begin{equation*}
	w_{\Pa} =  f_{013} + f_{024} + f_{078} + f_{125} + f_{168} + f_{267} + f_{346} + f_{357} + f_{458}.
	\end{equation*}
	This point lies in $\TGrU{3}{9}$ as it is the coordinatewise $t$-adic valuation of the Pl\"ucker coordinates of the $\QQ(\!(t)\!)$-valued matrix 
	\begin{equation*}
	\begin{pmatrix}
	1 & 0 & 0 & 2  & 1 & t  & 1+t  & -1   & 1 \\
	0 & 1 & 0 & -3 & t & 1  & -1   &  1-t^2 & 1 \\
	0 & 0 & 1 & 2t & 3 & -2 &  1   &  1   & 1+t
	\end{pmatrix}
	\end{equation*}
	The adjacency graph $\Gamma_w$ is a star tree whose central node is $v_{Q_{M_{\Pa}}}$ and has a leaf vertex $v_{Q_{M(ijk)}}$ for each $\{i,j,k\}\in \cL(M_{\Pa})$. The edge adjacent to  $v_{Q_{M(ijk)}}$ corresponds to $M(ijk)'$. A computation in affine coordinates yields $\dim \Gr_{M_{\Pa}} = 10$, hence  $\dim \Gr_w = 19$ by \Cref{prop:closedImmersionIso}. 
\end{example}
\begin{question}
	Does the analog of  \Cref{prop:3nDimension} hold for $\mat{3}{8}$-matroids?
\end{question}

Next, we discuss  the general behavior of the maps $\varphi_{M,M'}\colon \Gr_{M} \rightarrow \Gr_{M'}$ for $M' \leq M$. Recall that $\varphi_{M,M'}$ is a SDC-morphism whenever $M'$ is a $\mat{3}{6}$-matroid, or $M$ is a $\mat{3}{7}$-matroid and $\Delta_{M'}$ is not a face of $\Delta(3,7)$. When $M'$ is a face of $\Delta(3,7)$, then $\varphi_{M,M'}$ may fail to be dominant.

\begin{example}\label{ex:Bad3_8}
	Let $M$ be the $\mat{3}{7}$-matroid with lines
	\begin{equation*}
	\cL(M) = \{ \{0,3,6\},\, \{1,4,6\},\, \{2,5,6\}\}.
	\end{equation*}
	and $M' = M_{[5]}$. In a projective realization of $M$, the lines $\overline{03}$, $\overline{14}$, and $\overline{25}$ all meet at the point $6$. Observe that $6$ becomes a loop in $M'$, and $M' \cong U(3,6)\oplus U(0,1)$. A projective realization of $U(3,6)$ does not require that  $\overline{03}$, $\overline{14}$, and $\overline{25}$ all meet at a common point. Such a condition defines a codimension 1 subscheme of $\GrU{3}{6}$. Therefore $\varphi_{M,M'}$ is not dominant. Extending the ground set by adding points in linear general position, we get non-dominant morphisms $\varphi_{M,M'}$	for any realizable $\mat{3}{n}$ matroid with $n\geq 7$. By adding elements to the ground set parallel to any of the $\{0,\ldots,6\}$, one may produce non-dominant morphisms $\varphi_{M,M'}$ for $n\geq 8$ where $\Delta_{M'}$ is not a face of $\Delta(3,n)$.   
\end{example}

We end with an example of an initial degeneration of  $\GrU{3}{9}$ that is reducible, proving \Cref{theoremC}. 
Consider the Perles matroid $P$ of nine points and nine lines 
\begin{align*}
\cL(P) = \{ &\{0,1,2,3\},  \{0,4,8\},   \{1,4,7\}, \{0,5,6\},\\
 &\{1,5,8\}, \{3,4,5\}, \{2,5,7\}, \{2,6,8\}, \{3,6,7\}\}.
\end{align*}
This is depicted in~\Cref{fig:higherGrassmannians}. First, we parameterize its thin Schubert cell. 
\begin{proposition}
	\label{pr:Perles1}
	The thin Schubert cell $\Gr_P$ is isomorphic to $X_P \times \GG_m^8$ and
	\[X_P \cong \Spec(\kk[z^{\pm}]/\langle z^2-z-1\rangle).\]
	In particular, $\Gr_P$ has two connected components. 
\end{proposition}

\begin{proof}
	As discussed in the beginning of \Cref{sec:logCanonical}, the maximal torus  $H\subset \PGL_9(\kk)$ acts freely on $\Gr_P$ and $\Gr_P \cong X_P \times H$.  To compute $X_P$, we use affine coordinates similar to~\Cref{construction}. Let $X$ be the matrix with   $(1,0,0)$, $(0,1,0)$, $(0,0,1)$ and $(1,1,1)$, in columns $0$, $1$, $4$ and $5$, respectively. The $H$-action allows us to set  one nonzero entry of each remaining column to be 1. A standard calculation yields
	\begin{equation*} X := \begin{pmatrix}
	1 & 0 & -z & 1 & 0 & 1 & 1-z & 0      & 1  \\
	0 & 1 & 1      & 1      & 0 & 1 & 1 & z & 0 \\
	0 & 0 & 0      & 0      & 1 & 1 & 1      & 1      & 1 
	\end{pmatrix}. \end{equation*} 
	where $z^2-z-1=0$. Therefore  $X_P$ has the required form. 
\end{proof}

Let $w_P\in\wedge^3\ZZ^9$ be the vector
\begin{equation}\label{eq:wp}
w_P= f_{012} + f_{013} + f_{023} + f_{123} + f_{048} + f_{147} + f_{158} + f_{257} + f_{268} + f_{367}.
\end{equation}

\begin{proposition}
	\label{pr:Perles2}
	The vector $w_P$ from~\Cref{eq:wp} induces a dominant open immersion 
	\begin{equation*}
	\tsc{w_P}
	\hookrightarrow   \Spec(\kk[z]/\langle z^2-z-1 \rangle) \times \GG_m^{18}.
	\end{equation*}
	In particular, $\Gr_{w_P}$ is smooth and has two connected components.
\end{proposition}

\begin{proof}
	The adjacency graph $\Gamma_{w_P}$ is a star tree on nine leaves with central node $v_{Q_P}$. For each $\lambda\in \cL(P)$, there is a leaf vertex $v_{Q_{M(\lambda)}}$ connected to $v_{Q_P}$, and $Q_P\cap Q_{M(\lambda)} = Q_{M(\lambda)'}$, where $M(\lambda)$ and $M(\lambda)'$ are defined in \Cref{eq:Mlambda}.  By~\Cref{pr:FromAllToCodim1Diagram},
	\begin{equation*}
	\tsc{w_P} \cong
	\tsc{P} \times_{\prod_{\lambda\in \cL(P)} \tsc{M(\lambda)'}} \prod_{\lambda\in \cL(P)} \tsc{M({\lambda})}.
	\end{equation*}
	The embedding in the statement will be obtained by combining  the associativity of fiber products with~\Cref{pr:Perles1}, ~\Cref{pr:limitContractGraph} and the following two claims.
	\begin{enumerate}[noitemsep]
		\item $\Gr_P\times_{\Gr_{M(\lambda)'}} \Gr_{M(\lambda)} \cong \Gr_P\times\GG_m$  for $\lambda  \in \cL(P)\smallsetminus\{\{0,1,2,3\}\}$, and
		\item $\Gr_P\times_{\Gr_{M(0123)'}} \Gr_{M(0123)}$ embeds into $ \Gr_P \times \GG_m^{2}$ as a dense open subscheme.  
	\end{enumerate}
	Both verification are similar. We show the second one since it is more subtle.  To simplify notation, let $M = M(0123)$ and $M' = M(0123)'$. We use affine coordinates as in \Cref{construction}. Let $X$ be the matrix of variables $x_{ij}$ such that the columns 0,1,4 form the identity.  
	Then $R_{M'}^x \cong S'^{-1}\ZZ[x_{ij}^{\pm}  \,|\, ij \neq 02, 12]$ and $R_{M}^x \cong S^{-1}R_{M'}^x[x_{20}^{\pm}, x_{21}^{\pm}]$, where $S'$ is generated by $X_{234}$ and $S$ is generated by $X_{023}, X_{123}$. From this we conclude that
	\begin{equation*}
	\Gr_P\times_{\Gr_{M'}} \Gr_{M} \cong \Spec (S^{-1}  R_P^x [x_{20}^{\pm}, x_{21}^{\pm}]).
	\end{equation*}
	Since the $x_{00}$, $x_{01}$, $x_{10}$, and $x_{11}$ are units in $R_P^x$, $X_{023}$ and $X_{123}$ are not zero divisors in the above ring. Therefore the natural map $R_P^x[x_{20}^{\pm}, x_{21}^{\pm}] \to  S^{-1}  R_P^x [x_{20}^{\pm}, x_{21}^{\pm}]$ is injective and (ii) holds.
\end{proof}

\begin{lemma}
	\label{lm:Perles3} 
	We have an isomorphism $\init_{w_P}\GrU{3}{9} \cong
	\tsc{w_P}$.
\end{lemma}

\begin{proof}
	Because $\psi_{w_P}:\inw_w\GrU{3}{9} \hookrightarrow \Gr_{w_P}$ is a closed immersion of affine schemes of the same dimension and $\Gr_{w_P}$ is reduced, it suffices to show that the image of $\psi_{w_P}$ meets the two connected components of $\Gr_{w_P}$ by \Cref{prop:closedImmersionIso}. 	 Consider the following $\kk(\!(t)\!)$-valued matrix
	\begin{equation*} 
	X(a) := \begin{pmatrix}
	1 & 0 & -a & 1   & 0 & 1   & 1-a  & t & 1  \\
	0 & 1 & 1  & 1+t & 0 & 1   & 1    & a & 3t \\
	0 & 0 & t  & 2t  & 1 & 1+t & 1+2t & 1 & 1+3t 
	\end{pmatrix}. 
	\end{equation*} 
	where $a=b,\bar{b}$ are the distinct solutions to $z^2-z-1 = 0$. Let $p_a$ be the Pl\"ucker vector of $X(a)$.  One may verify that the coordinatewise valuation of $p_a$ is $w_P$. The exploded tropicalization $\mathfrak{Trop}(p_a)$ is an element of $\inw_w\GrU{3}{9}$ \cite[Lemma~3.2]{Payne}, and  $\psi_{P,w}$ maps $\mathfrak{Trop}(p_b)$, $\mathfrak{Trop}(p_{\bar{b}})$ to different connected components of $\Gr_{w_P}$, as required.
\end{proof}

\begin{proof}[Proof of~\Cref{theoremC}]
\Cref{pr:Perles2} and \Cref{lm:Perles3} yield a dominant open immersion $\init_{w_P}\GrU{3}{9} \hookrightarrow \Spec(\kk[z]/\langle z^2-z-1 \rangle) \times \GG_m^{18}$. Therefore $\init_{w_P}\GrU{3}{9}$ is smooth with two connected components. 
\end{proof}

\appendix

\section{Some functorial properties of SDC-morphisms}
\label{appendix:SDC}

Throughout this section, all $\kk$-schemes are Noetherian and of finite-type over $\kk$. In addition, all morphisms of $\kk$-schemes are also of finite-type. Recall from the beginning of~\Cref{sec:smoothnessOfTSC} that a SDC-morphism of $\kk$-schemes is one that is smooth and dominant with connected fibers. In this section, we will catalog properties of SDC-morphisms used throughout the paper.  First, we discuss how to deduce smoothness or connectedness of a $\kk$-scheme $X$ from properties of a morphism $X\to Y$ and $Y$. 

\begin{proposition} 
	\label{prop:SDCOpen}
	Let $X,Y$ be $\kk$-schemes as above. 
	\begin{enumerate}[noitemsep]
		\item  If $f:X\to Y$ is a dominant morphism with connected fibers and $Y$ is irreducible, then $X$ connected.
		\item If $f:X\to Y$ is a SDC-morphism and $Y$ is smooth and irreducible, then so is $X$.
	\end{enumerate}
\end{proposition}

\begin{proof}
	 Let $V$ be the image of $f$. Since $f$ is dominant and $Y$ is irreducible, $V$ is also irreducible. Therefore, $f:X\to V$ is a surjective morphism with connected fibers. We conclude that $X$ is connected by ~\cite[\href{https://stacks.math.columbia.edu/tag/0378}{Tag 0378}]{stacks-project}. Finally, (2) follows easily from (1). 
\end{proof}

Next, we explore how SDC-morphisms behave under base change. This proposition is crucial in the proof of \Cref{theoremA} as it will allow us to  deduce smoothness and irreducibility of initial degenerations by studying thin Schubert cells and the morphisms between them.

\begin{proposition}
	\label{prop:connectedFiberProduct}  Suppose we have a  pullback diagram 
	\begin{equation*}
	\xymatrix
	{{W\times_Z X}\ar[r]^-{h'} \ar[d]_-{f'} & X \ar[d]^f\\
		{W} \ar[r]^{h} & Z,}
	\end{equation*} 
	and that $W\times_Z X$ is nonempty. The following properties hold:
	\begin{enumerate}[noitemsep]
		\item If $f$ is smooth and $W$ is a smooth $\kk$-scheme, $W\times_Z X$ is smooth. 
		\item If $f$ is a SDC-morphism and $W$ is irreducible, then $f'$ is also a SDC-morphism.	
		\item If $f$ is a SDC-morphism and $W$ is smooth and irreducible, then $W\times_Z X$ is smooth and irreducible. 
	\end{enumerate}
\end{proposition}
\begin{proof} 
	To simplify notation, set $V:=W\times_Z Y$.   If $f$ and $W \to \Spec \kk$ are smooth morphisms, then so is $V \to \Spec \kk$  since smoothness is preserved under composition and base-change. This proves (1).

	Now suppose $f$ is a SDC-morphism and $W$ is irreducible. So $f': V\to W$ is smooth, in particular flat. By~\cite[Exercise~III.9.1]{Hartshorne} $f'$ is also open. This means that $f'(V)$ is a nonempty open subscheme of $W$, which is dense by the irreducibility of $W$. For $w\in f'(V)$, the fiber $V_{w}$ is nonempty and isomorphic to $X_{h(w)}$, which is connected, hence (2). Statement (3) follows from this and \Cref{prop:SDCOpen}(2). 
\end{proof}

\begin{proposition} 
	\label{prop:SDCcomposition} 
	SDC-morphisms satisfy the following.
	\begin{enumerate}[noitemsep]
		\item A dominant open immersion $U \hookrightarrow X$ is a SDC-morphism.
		\item If $f:X\to Y$ and $g:Y\to Z$ are SDC-morphisms, then $gf:X\to Z$ is a SDC-morphism. 	
	\end{enumerate}
\end{proposition}

\begin{proof}
	Statement (1) is clear, so consider (2).  It is well known that smoothness and dominance are preserved under composition, so we need only show that $gf$ has connected fibers. Let $z\in Z$, and $X_z$ (resp. $Y_z$) be the scheme-theoretic fiber of $gf$ (resp. $g$) over $z$. Let $f_z:X_z\to Y_z$  be the morphism obtained by pulling back $f$ along the inclusion $Y_z \to Y$. Since $Y_z$ is smooth and connected, it is irreducible. By \Cref{prop:connectedFiberProduct}(2), $f_z$ is a SDC-morphism. Therefore $X_z$ is connected by \Cref{prop:SDCOpen}(1), as required. 
\end{proof}

Many of the limits that appear in this paper come from graphs in the following way. Let $\cC$ be a category that has finite limits (e.g., $\cC = \kSch$,  the category of $\kk$-schemes), and $G$ a connected graph, possibly with loops or multiple edges. We regard each edge $e\in E(G)$ as a pair of half-edges. Let us define a quiver $Q(G)$.  The set of vertices of $Q(G)$ is $V(G) \cup E(G)$; we write $q_{v}$ ($v\in V(G)$), resp. $q_{e}$ ($e\in E(G)$), for the corresponding vertex of $Q$. For every half edge $h$ of $e$ incident to $v$, there is an arrow $q_h:q_v\to q_e$. In particular, if $e$ is a loop edge, then there are two arrows from $q_v$ to $q_e$. Viewing $Q(G)$ as a category in the usual way, a \textit{diagram} of type $Q(G)$ in a category $\cC$ is a functor $X:Q(G) \to \cC$. We write $X_{v} = X(q_{v})$, $X_e = X(q_e)$, $\varphi_h = X(q_h)$ and $X_G = \varprojlim_{Q(G)}X$. For example, \Cref{fig:graphDiagram} exhibits a graph and its corresponding diagram.

\begin{figure}[tb]
	\begin{center}
		\begin{minipage}[c]{0.4\textwidth}
			\includegraphics[scale=0.9]{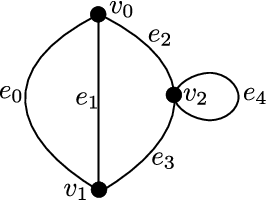}
		\end{minipage}
		\begin{minipage}[c]{0.5\textwidth}
			\begin{equation*}
			\begin{aligned}
			\xymatrix{
				& X_{v_0} \ar[dl] \ar[d] \ar[r]  & X_{e_2}  & \\
				X_{e_0}   & X_{e_1} &  X_{v_{2}} \ar[u] \ar[d] \ar@/^/[r] \ar@/_/[r] & X_{e_4} \\
				& X_{v_1}  \ar[ul] \ar[u] \ar[r]  & X_{e_3}  & 
			}
			\end{aligned}
			\end{equation*}
		\end{minipage}
	\end{center}
	
	\caption{A graph and its associated diagram}\label{fig:graphDiagram}
\end{figure}

\begin{example}
	\label{ex:diagramDualGraph}
	Let $\Gamma_{M,w}$ be the adjacency graph to a matroid subdivision $\Delta_{M,w}$.
	Let $M_v$, resp. $M_e$, denote the matroid corresponding to the vertex $v$, resp. edge $e$, of $\Gamma_{M,w}$, and $\varphi_{M_v,M_e}:\Gr_{M_{v},M_e}$ whenever $e$ is incident to $v$. The data of $\Gr_{M_v} , \Gr_{M_e}$, and $\varphi_{M_v,M_e}$ defines a diagram of type  $Q(\Gamma_{M,w})$ in $\kSch$.
\end{example}

Now, let us consider how this construction behaves with respect to contracting a connected subgraph. Let $F$ be a connected subgraph of $G$, and $G/F$ the graph obtained by contracting $F$ to a single vertex $v_{F}$.   Let $X_{F} = \varprojlim_{Q(F)} X$  and $\xi_v:X_F \to X_v$, $\xi_{e}:X_F \to X_e$ the structure morphisms. 
Set $Y_{v_{F}} = X_F$, and $Y_{v} = X_v$ for the remaining $v$ in $V(G/F)$. Similarly, let $Y_e = X_e$ for the edges $e\in E(G/F)$.   If $h$  is a half edge in $G/F$ incident $v_F$, set $\psi_h = \varphi_h\xi_v$. Otherwise, let $\psi_h = \varphi_h$. The data $(Y_{v},Y_e,\psi_h)$ defines a diagram $Y$ of type $Q(G/F)$.

\begin{proposition} 	\label{pr:limitContractGraph}
	We have an isomorphism 
	\begin{equation*}
	\varprojlim_{G} X \cong \varprojlim_{G/F} Y.
	\end{equation*}	
\end{proposition}

\begin{proof}
	
	To simplify notation, set $Y_{G/F} = \varprojlim_{G/F} Y$. Let  $\lambda_v: Y_{G/F} \to Y_v$, $\lambda_e: Y_{G/F} \to Y_e$ denote the  structure morphisms of this limit.	
	We show that $Y_{G/F}$ satisfies the universal property for $\varprojlim_{G} X$. First, we must define morphisms $\alpha_v: Y_{G/F} \to X_v$, $\alpha_e: Y_{G/F} \to X_{e}$ that commute with each $\varphi_h$. This is achieved by setting  $\alpha_v = \xi_v\lambda_{v_F}$, (resp. $\alpha_e = \xi_e\lambda_{v_F}$) when $v\in V(F)$ (resp. $e\in E(F)$), and   $\alpha_v = \lambda_v$, $\alpha_e = \lambda_e$ otherwise.  One may verify that $\varphi_h \alpha_v = \alpha_e$.

	Now suppose that we have a collection of morphisms  $\theta_v: M\to X_v$ and $\theta_e:M\to X_e$ such that $\varphi_h\theta_v = \theta_e$ for every $q_h:q_v\to q_e$ in $Q(G)$. We will show that there is a unique morphism $\theta: M \to Y_{G/F}$ such that 
	\begin{equation}\label{eq:alphaTheta}
	\theta \alpha_{v} = \theta_v \text{ and } \theta \alpha_{e} = \theta_e	
	\end{equation}
	for all $v\in V(G)$ and $e\in E(G)$, respectively.
	By the universal property of $X_F$, there is a unique morphism $\theta_{v_F}: M \to Y_{v_F}$ such that $\xi_v\theta_{v_F} = \theta_v$ and $\xi_e\theta_{v_F} = \theta_e$. If $h$ is a half edge in $G/F$ incident to $v_F$, then 
	\begin{equation*}
	\psi_h \theta_{v_F} = \varphi_h\xi_v\theta_{v_F} = \varphi_h\theta_v = \theta_e.
	\end{equation*}
	Otherwise, $\psi_h \theta_v = \theta_e$ since $\psi_h = \varphi_h$. By the universal property of $Y_{G/F}$, there is a unique morphism $\theta:M\to Y_{G/F}$ satisfying $\lambda_v \theta= \theta_v$ and $\lambda_e \theta= \theta_e$. 
	
	Now we establish the equalities in~\Cref{eq:alphaTheta}. 	 When $v\in V(F)$,  
	\begin{equation*}
	\alpha_v\theta  = \xi_v \lambda_{v_F}\theta = \xi_v \theta_{v_F} = \theta_v.
	\end{equation*} 
	A similar argument shows that $\alpha_e\theta = \theta_e$ when $e\in E(F)$. The cases where $v\in V(G)\setminus V(F)$ or $e\in  E(G)\setminus E(F)$ follow from the identifications $\alpha_v = \lambda_v$ and $\alpha_e = \lambda_e$. 
	Finally, the uniqueness of $\theta$ follows from the uniqueness of $\theta_F$ and the universal property of $Y_{G/F}$. 
\end{proof}

\begin{proposition}
	\label{lm:limitTree}
	Suppose $G$ is a tree, $X$ a diagram of type $Q(G)$ in $\kSch$ such  $X_v$ and $X_e$ are smooth and irreducible $\kk$-schemes, and for each half edge $h$, $X(h):X_v\to X_e$ is a SDC-morphism. Then $X_G$ is smooth and irreducible. Moreover, 
	\begin{equation*}
	\dim X_G = \sum_{v\in V(G)} X_v - \sum_{e\in E(G)} X_e.
	\end{equation*}
\end{proposition}

\begin{proof}
	We proceed by induction on the number of vertices. When $G$ consists of a single vertex, there is nothing to show. Now suppose that the lemma is true for all trees with fewer vertices than $G$. Let $w$ be a one valent vertex of $G$, $e$ the adjacent edge, and $G'$ the graph consisting of the remaining vertices and edges. By~\Cref{pr:limitContractGraph}, 
	\begin{equation*}
	X_G \cong X_w \times_{X_{e}} X_{G'}.
	\end{equation*}
	It is smooth and irreducible by \Cref{prop:connectedFiberProduct} and the inductive hypothesis. Because $X_w\to X_e$ is smooth of relative dimension $\dim X_{w} - \dim X_e$, so is $X_G  \to X_G'$, and therefore 
	\begin{equation*}
	\dim X_G = \dim X_w - \dim X_e + \dim X_{G'}
	\end{equation*}
	by \cite[Corollary~9.6]{Hartshorne}. By the inductive hypothesis, we get the required formula for $\dim X_G$.
	\end{proof}

\begin{remark}
	\label{rmk:equalizerCounterexample}
	An arbitrary finite limit over a diagram of smooth and irreducible $\kk$-schemes in which every morphism is SDC-need not be irreducible. 
	Let $h =x^2-y^2+x+\frac{1}{4}$ and  $X = \Spec((h)^{-1} \kk[x,y])$.  Define two morphisms $f,g:X \to \Spec \kk[z]$ by:
	\begin{align*}
	f^{\#}(z) = x^2-y^2 + x && g^{\#}(z) = x.
	\end{align*}
	One may verify that $f$ and $g$ are SDC-morphisms between smooth and irreducible $\kk$-schemes. However, the equalizer of $f$ and $g$ is 
	\begin{equation*}
	\Spec((h)^{-1}\kk[x,y]/ \langle x^2-y^2 \rangle)
	\end{equation*}
	which is neither smooth nor irreducible. 
\end{remark}

We end with a proposition on when a closed immersion of affine schemes is an isomorphism. 

\begin{proposition}
	\label{prop:closedImmersionIso}
	Suppose $\varphi:X\hookrightarrow Y$ is a closed immersion of affine schemes, and $Y$ is integral. If $\dim X = \dim Y$, then $\varphi$ is an isomorphism.
\end{proposition} 

\begin{proof}
	Let $n$ be the Krull dimension of $X$ and $Y$. Because $\varphi$ is a closed immersion and $Y$ is integral, the induced morphism on rings is of the form $\varphi^{\#}:R \to R/I$ for some integral domain $R$ and ideal $I\subset R$. A maximal chain of prime ideals $\mathfrak{p}_0 \subsetneq \mathfrak{p}_1\subsetneq \cdots \subsetneq \mathfrak{p}_n$ in $R/I$ lifts to a maximal chain of prime ideals  $\mathfrak{q}_0 \subsetneq \mathfrak{q}_1\subsetneq \cdots \subsetneq \mathfrak{q}_n$ in $R$ with $I\subset \mathfrak{q_0}$. Because $R$ is an integral domain, $\mathfrak{q_0} = \langle 0\rangle$.  So $I = \langle 0 \rangle$, and therefore $\varphi^{\#}$ is the identity.   
	\end{proof}

\section{Data for \Cref{lm:centerlimit}}
\label{appendix:37Data}

In \Cref{table:uniqueIdeals}, we list the ideals that appear as $J_{M,w}^x$ for subdivisions of $\Delta_{M,w}$ such that $M$ is a simple, connected $\kk$-realizable $\mat{3}{7}$-matroid and $\Gamma_{M,w}$ has no leaves, as in the proof of \Cref{lm:centerlimit}. We consider all of these as ideals in the ring 
\begin{equation*}
\kk[x_{ij}^{\pm} \ | \ 0\leq i \leq 2,\ 0\leq j\leq 3].
\end{equation*}
We write $\kk[x_{ij}^{\pm}]$ for short. Many of the polynomials that appear are of the form $X_{ij,k\ell} := x_{ik}x_{j\ell} - x_{i\ell}x_{jk}$. In the second column, we list variables that may be eliminated to produce an isomorphism of $\kk[x_{ij}^{\pm}]/J_{M,w}^x$ with a Laurent polynomial ring.  
For example, consider the last row. In this case, 
\begin{equation*}
J_{M,w}^x = \langle X_{01,23},X_{02,03},X_{12,12},X_{12,02},X_{12,01} \rangle.
\end{equation*}
We use $X_{01,23}$ to solve for $x_{02}$, $X_{02,03}$ for $x_{03}$, $X_{12,12}$ for $x_{11}$, and finally $X_{12,02}$ for $x_{22}$. This produces an isomorphism $\kk[x_{ij}^{\pm}]/J_{M,w}^x \to \kk[x_{ij}^{\pm} \, | \, ij\neq 02,03,11,22]$.

\begin{table} 	
	\centering
	\begin{tabular}{|l|l| } 
		\hline
		$\!$Ideals & $x_{ij}$ to eliminate \\
		\hline \hline
$\!X_{12,01}$      &   $x_{21}$  \\  \hline
$\!X_{01,12}$      &   $x_{12}$  \\  \hline
$\!X_{12,12}$      &   $x_{22}$  \\  \hline
$\!X_{02,23}$      &   $x_{23}$  \\  \hline
$\!X_{01,03}$      &   $x_{13}$  \\  \hline
$\!X_{01,23}$      &   $x_{13}$  \\  \hline
$\!X_{12,23}$      &   $x_{23}$  \\  \hline
$\!X_{01,02}$      &   $x_{12}$  \\  \hline
$\!x_{02}x_{10}x_{21}+x_{00}x_{11}x_{22}$      &  $x_{22}$  \\   \hline
$\!x_{00}x_{12}x_{21}+x_{01}x_{10}x_{22}$      &  $x_{22}$  \\   \hline
$\!x_{00}x_{13}x_{22}+x_{02}x_{10}x_{23}$      &  $x_{23}$  \\   \hline
$\!x_{02}x_{10}x_{21}-x_{00}x_{12}x_{21}-x_{01}x_{10}x_{22}$     &  $x_{22}$   \\  \hline
$\!x_{02}x_{11}x_{20}-x_{01}x_{12}x_{20}-x_{00}x_{11}x_{22}$      &  $x_{22}$  \\   \hline
$\!x_{03}x_{10}x_{22}-x_{00}x_{13}x_{22}-x_{02}x_{10}x_{23}$      &  $x_{23}$  \\   \hline
$\!x_{03}x_{10}x_{22}-x_{00}x_{13}x_{22}+x_{00}x_{12}x_{23}$      &  $x_{23}$  \\   \hline
$\!x_{02}x_{10}x_{21}-x_{00}x_{12}x_{21}-x_{01}x_{10}x_{22}+x_{00}x_{11}x_{22}$      &   $x_{11}$  \\  \hline
$\!X_{02,23},X_{12,01}$      &  $x_{21}, x_{23}$   \\  \hline
$\!X_{02,02},X_{12,01}$      &  $x_{21}, x_{22}$   \\  \hline
$\!X_{01,03},X_{12,01}$      &  $x_{13}, x_{21}$   \\  \hline
$\!X_{02,03},X_{12,01}$      &  $x_{21}, x_{23}$   \\ \hline
$\!X_{02,13},X_{01,02}$      &  $x_{12}, x_{23}$   \\  \hline
$\!X_{02,12},X_{12,01}$      &  $x_{01}, x_{22}$   \\  \hline
$\!X_{02,13},X_{12,01}$      &  $x_{10}, x_{23}$   \\  \hline
$\!X_{01,13},X_{01,02}$      &  $x_{12}, x_{13}$   \\  \hline
$\!X_{12,13},x_{00}x_{12}x_{21}+x_{01}x_{10}x_{22}$      &  $x_{22}, x_{23}$   \\  \hline
$\!X_{01,13},x_{02}x_{10}x_{21}-x_{00}x_{12}x_{21}-x_{01}x_{10}x_{22}$      &   $x_{13}, x_{22}$   \\ \hline
$\!X_{12,13},x_{02}x_{10}x_{21}-x_{01}x_{10}x_{22}+x_{00}x_{11}x_{22}$      &  $x_{02}, x_{23}$    \\ \hline
$\!X_{01,12},x_{03}x_{10}x_{22}-x_{00}x_{13}x_{22}-x_{02}x_{10}x_{23}$      &  $x_{12}, x_{23}$    \\ \hline
$\!X_{02,03},X_{01,12},X_{12,01}$      &   $x_{12},x_{21},x_{23}$   \\ \hline
$\!X_{01,13},X_{02,02},X_{12,01}$      &   $x_{13},x_{22},x_{21}$   \\ \hline
$\!X_{02,13},X_{02,02},X_{12,01}$      &   $x_{10},x_{22},x_{23}$   \\ \hline
$\!X_{12,12},X_{12,02},X_{12,01}$      &   $x_{22}, x_{21}$   \\ \hline
$\!X_{12,03},X_{01,23},x_{02}x_{11}x_{20}+x_{01}x_{10}x_{22}-x_{00}x_{11}x_{22}$      &   $x_{00}, x_{12}, x_{23}$  \\  \hline
$\!X_{01,03},X_{02,02},-X_{12,01}-x_{02}x_{13}x_{21}+x_{03}x_{11}x_{22}$      &  $x_{00}, x_{01}, x_{13}$   \\  \hline
$\!X_{02,03},X_{12,12},X_{12,02},X_{12,01}$      &  $x_{11},x_{22},x_{23}$   \\  \hline
$\!X_{02,13},X_{01,12},X_{12,01},x_{03}x_{12}x_{20}-x_{02}x_{10}x_{23}$      &     $x_{12},x_{21},x_{23}$    \\   \hline
$\!X_{02,02},X_{12,01},x_{01}x_{13}x_{22}+x_{02}x_{11}x_{23},x_{01}x_{13}x_{20}+x_{00}x_{11}x_{23}\!$  
   &  $x_{00}, x_{10}, x_{22}$   \\  \hline
$\!X_{12,23},X_{02,13},X_{02,02},X_{12,01},$      &   $x_{00}, x_{13}, x_{20}, x_{22},\!$ \\
$x_{00}x_{03}x_{11}x_{12}-x_{01}x_{02}x_{10}x_{13}$& $x_{23}$  \\  \hline
$\!X_{01,23},X_{02,03},X_{12,12},X_{12,02},X_{12,01}$      &  $x_{02}, x_{03}, x_{11}, x_{22}\!$   \\  \hline
	\end{tabular}
	\caption{Here are the unique ideals nontrivial ideals that appear in the proof of \Cref{lm:centerlimit}}
	\label{table:uniqueIdeals}
\end{table}

\newpage

\section{Maps between thin Schubert cells and inverse limits (written by  Mar\'ia Ang\'elica Cueto)} 
\label{appendix:MAC}

In this appendix, we discuss how to reduce the study of geometric properties of thin Schubert cells to the case of  simple and connected matroids.  Because the only simple rank 2 matroid is $U(2,n)$, this analysis gives us  a complete understanding of $\Gr_M$ in the rank 2 case, and simplifies the study of rank $3$ matroids in Sections \ref{sec:smoothnessOfTSC} and \ref{sec:smoothnessMorphisms}. In the following subsection, we show that the limit of thin Schubert cell $\Gr_{M,w}$ induced by a matroid subdivision $\Delta_{M,w}$ depends only on the adjacency graph of $\Delta_{M,w}$. This allows one to apply the results from \Cref{appendix:SDC} to study $\Gr_{M,w}$ as in \Cref{sec:smoothIrreducibleDeg}.

\subsection{Reduction to simple and connected matroids}

The following two Lemmas demonstrate that thin Schubert cells are compatible with decomposition into connected components and  removal of loops and parallel elements.  \Cref{lem:reduceToConnected} appears \cite[Proposition~9.4]{KatzMatroid} without proof, and \Cref{lem:parallelTSC} will appear in an upcoming paper \cite{cueto}. For the reader's convenience, we sketch their proofs.

\begin{lemma}\label{lem:reduceToConnected}
If $M = M_1\oplus M_2$, then  $\Gr_{M} \cong \Gr_{M_1}\times \Gr_{M_2}$. In particular,  $\Gr_{M} \cong \Gr_{M|T}$ where $T\subset [n]$ is the set of non-loop elements.
\end{lemma}

\begin{proof}
	Suppose $X_1$ and $X_2$ are matrices giving rise to the rings $R_{M_1}^{x}$, $R_{M_2}^{x}$ as in \Cref{construction}.  Let $X$ be the block matrix with $X_1$ and $X_2$ on the diagonal. Then $R_M^x \cong R_{M_1}^{x} \otimes R_{M_2}^{x}$. The second statement follow from  $M \cong M|T \oplus U(0,|T|)$.
\end{proof}

Given a matroid $M$, we define a simple matroid by removing loops and parallel elements in the following way. Let $\eta_1,\ldots,\eta_k$ be the rank 1 flats of $M$, choose nonloop elements $s_i\in \eta_i$ and set $S = \{s_1,\ldots, s_k\}$. Then $M|S$ is a simple matroid. Let $\ell$ be the number of loops in $M$.

\begin{lemma}
	\label{lem:parallelTSC}
	We have an isomorphism $\Gr_{M} \cong \Gr_{M|S} \times \GGm^{n-k-\ell}$. 
\end{lemma}

\begin{proof}
	By \Cref{lem:reduceToConnected}, we may assume that $M$ has no loops. 
	Suppose  $i,j$ are parallel in $M$, then $\Gr_{M} \cong \Gr_{M_{[n]\setminus i}} \times \GGm$.
	 Suppose $i$ and $j$ are parallel, and let $\mu_1,\mu_2\in {[n]\choose d-1}$ such that $\mu_k\cup\{i\}$ and $\mu_k\cup\{j\}$  are bases of $M$ for $k=1,2$. The quadratic generator from \Cref{eqn:lambdamui} yields
	 \begin{equation*}
	 \qp{M}{\mu_1\cup\{i,j\}}{\mu_2} = p_{\mu_1\cup i} p_{\mu_2\cup j} - p_{\mu_1\cup j} p_{\mu_2\cup i}
	 \end{equation*}  
	 This means that $p_{\mu\cup i}/p_{\mu\cup j}$ is independent of $\mu$. At the level of rings, the desired isomorphism $R_{M_{[n]\setminus i}} \otimes \kk[t^{\pm}] \cong R_{M} $ is given by
	\begin{align*}
	R_{M_{[n]\setminus i}} \otimes \kk[t^{\pm}] &\longrightarrow R_{M}   \\
	p_{\lambda}\otimes 1 &\mapsto p_{\lambda} \;\;\;\;\;\;\;\; \;\;\;\;\;\;\; \text{ if } i\notin \lambda,  \\
	1 \otimes t &\mapsto p_{\mu \cup i}/p_{\mu\cup j}  \;\;\;\; \text{ if } \lambda = \mu \cup \{i\}.
	\end{align*} 
	The Lemma now follows by induction on the number of parallel elements in $M$.
\end{proof}

Because the uniform matroid $U(2,n)$ is the only simple $(2,[n])$-matroid, and affine coordinates realize $\GrU{2}{n}$ as a open subscheme of an algebraic torus, we have the following. 

\begin{proposition}
	\label{prop:rk2smooth}
	If $M$ is a rank 2 matroid then
	\begin{equation} 
	\label{eqn:rk2decomp}
	\Gr_M \cong \GrU{2}{k} \times \GGm^{n-k-\ell}.
	\end{equation}
	where $k$ is the number of rank 1 flats and $\ell$ the number of loops. In particular, $\Gr_{M}$ is smooth and irreducible. 
\end{proposition}

Next, we  show the morphisms $\varphi_{M,M'}:\Gr_{M} \to \Gr_{M'}$ are compatible with the following operations: decomposition of matroids into connected components, removal of loops and parallel elements, and duality.  This will allow us to restrict our attention to pairs $M'\leq M$ where $M$ is simple, connected, and $d = r_M([n]) \leq \lfloor n/2\rfloor$.

\begin{lemma}\label{lm:connectedAndVarphi} If $M'\leq M$ and $M=M_1 \oplus M_2$, then $M' = M'_1 \oplus M'_2$ with $M'_i\leq M_i$ for $i=1,2$. Furthermore, $\varphi_{M,M'} = \varphi_{M_1,M'_1} \times \varphi_{M_2,M'_2}$.
\end{lemma}
\begin{proof} Recall that $\Delta_M= \Delta_{M_1} \times \Delta_{M_2}$ if and only if $M=M_1\oplus M_2$. Thus,  a face of $\Delta_M$ must be of the form  $\Delta_{M_1'}\times \Delta_{M_2'}$ for $M_i'\leq M_i$ for $i=1,2$ so $M'=M_1'\oplus M_2'$. The statement regarding $\varphi_{M,M'}$ follows by combining this decomposition with~\Cref{prop:facemapsA} and \Cref{lem:reduceToConnected}.
\end{proof}

\begin{lemma}\label{lm:simpleReduction}  If $M' \leq M$, then we have $M'|S \leq M|S$ and the restrictions fit into the commutative diagram:
	\begin{equation}
		\begin{aligned}
			\label{eq:diagramSimplematroids}\xymatrix{ \tsc{M} \ar[r]^-{\cong} \ar[d]_{\varphi_{M,M'}} & \tsc{M|S} \times \GG_m^{n-k-\ell}\ar[d]^{\varphi_{M|S,M'|S}\times id} \\
				\tsc{M'} \ar[r]^-{\cong}  &  \tsc{M'|S} \times \GG_m^{n-k-\ell}.
			}
		\end{aligned}
	\end{equation}
\end{lemma} 
\begin{proof}   The top horizontal map in~\eqref{eq:diagramSimplematroids} arises from the isomorphism  $\tsc{M}\simeq \tsc{M|S} \times \GG_m^{n-k-\ell}$ described in~\Cref{lem:parallelTSC}. Since $M'\leq M$, rank-one flats in $M$ yield rank-one flats  in $M'$ we have $M'|S \leq M|S$. The same lemma yields $\tsc{M} \cong \tsc{M'|S} \times \GG_m^{n-k-\ell}$. This determines  the bottom horizontal map. 
\end{proof}

Now we ensure the compatibility of $M'\leq M$ with the duality operation and the isomorphism 
\begin{equation*}
\psi:\Gr(d,n)\to \Gr(n-d,n) \hspace{10pt} (p_\beta)_{\beta}\mapsto ((-1)^{\sign(\beta,\beta^c)} p_{\beta^c})_{\beta^c}
\end{equation*}
induced from $\PP(\wedge^d\kk^n) \cong \PP(\wedge^{n-d}\kk^n)$. Here, $(\beta,\beta^c)$ is a permutation of $S_{n}$ in one-line notation and $\beta^c = [n]\smallsetminus \beta$. On affine patches, the correspondence for matrices is explicit: for example, a $d\times n$ matrix $(I_d|X)$ in $\{p_{[d]}\neq 0\}$ is identified with the $(n-d)\times n$ matrix $(-X^{t}|I_{n-d})$  in $\{p_{[d]^c}\neq 0\}$.

\begin{lemma}\label{lm:DualsRespectFaces} If $M'\leq M$ then $(M')^*\leq M^*$ and $\varphi_{M^*,(M')^*} = \psi \circ \varphi_{M,M'} \circ \psi^{-1}$.
\end{lemma}
\begin{proof} By definition, $\Delta_{M^*} = \conv(\{\mathbf{1}-e_\beta \colon \beta \in \cB(M)\}) = \mathbf{1} - \Delta_M$. In particular, if $\Delta_{M'}$ is a  face of $\Delta_M$, then $\Delta_{(M')^*} = (\mathbf{1} - \Delta_{M'}) \prec  (\mathbf{1} - \Delta_{M}) = \Delta_{M^*}$, as required. The isomorphism $\psi$ identifies each $p_\beta \in \kk[\GrC{d}{n}]$ with $p_{\beta^c} \in \kk[\GrC{n-d}{n}]$. The expression $\varphi_{M^*,(M')^*} = \psi \circ \varphi_{M,M'} \circ \psi^{-1}$ follows from this observation.
\end{proof}

As an application, we prove that $\varphi_{M,M'}$ is a SDC-morphism whenever $M=U(d,n)$ or  $M$ is a rank $2$ matroid.

\begin{proposition}\label{pr:uniform}
	For any $M'\leq M:=U(d,n)$, the map  $\varphi_{M,M'}\colon \GrU{d}{n}\to \tsc{M'}$ is a SDC-morphism. 
\end{proposition}

\begin{proof} By \Cref{prop:SDCcomposition}, it suffices to show that $\varphi_{M,M'}$ is a SDC morphism when $M'\lessdot M$. The nondegenerate subsets of $M$ are of the form $\{i\}$ or $[n]\setminus i$ for some $i \in [n]$. If $M' = M_{\{i\}}$, then $(M')^* = M^*_{[n]\setminus i}$ where $M^* \cong U(n-d,n)$. By \Cref{lm:DualsRespectFaces}, it suffices to consider just $M'=M_{[n]\setminus \{i\}}$.  
In this case, $R_{M'}^x = (S_{M'}^x)^{-1}B_{M'}$ and $R_{M'}^x = (S_{M}^x)^{-1}R_{M'}^x[x_{i,n-4}^{\pm} \, | \, i\in [d]\,]$. Therefore $\Gr_{M} \subset \GGm^{d\times (n-d)}$, $\Gr_{M} \subset \GGm^{d\times (n-d-1)}$ as open subvarieties, and $\varphi_{M,M'}$ is induced by a coordinate projection $\GGm^{d\times (n-d)} \to \GGm^{d\times (n-d-1)}$, which is clearly a SDC-morphism.  The result now follows from  \autoref{prop:SDCcomposition}.
\end{proof}

\begin{proposition}\label{prop:rk2smoothmap}
	For $(2,[n])$-matroids $M'\leq M$, $\varphi_{M,M'} $ is a SDC morphism. 
\end{proposition}

\begin{proof}
	Because every simple rank $2$ matroid is uniform, the Proposition follows from Lemmas \ref{lm:connectedAndVarphi}, \ref{lm:simpleReduction}, and \ref{pr:uniform}.
\end{proof}

Our final result in this subsection say that the reduction to simple matroids as above is compatible taking initial degenerations and inverse limits.

\begin{proposition}
	\label{lm:simpleIsEnoughInitialDeg} 
	Fix  $w\in \TGr_M$ and let $\tilde{w}$ be the projection of $w$ to $\RR^{\cB(M|S)}/\rspan{\mathbf{1}}$. Then $\tilde{w} \in \TGr_{M|S}$ and
	\begin{align*} \init_{w} \tsc{M} \simeq  \init_{\tilde{w}} \tsc{M|S} \times \GGm^{n-k-\ell} && \Gr_{M,w} \cong  \Gr_{M|S,\tilde{w}} \times \GGm^{n-k-\ell}
	\end{align*}
\end{proposition}

\begin{proof}  The assertion on initial degenerations follows from the fact that the isomorphism $\Gr_M \cong \Gr_{M|S} \times \GGm^{n-k-\ell}$ from \Cref{lem:parallelTSC} is induced by a monomial map on coordinate rings.  The isomorphism of limits follows from this Lemma and the description of the coordinate ring of $\Gr_{M,w}$ in \Cref{prop:limitideal}.
\end{proof}

\subsection{Limits of thin Schubert cells via adjacency graphs}

\smallskip

Recall that the matroid subdivision $\Delta_{M,w}$ yields a system of maps $\varphi_{M_Q,M_{Q'}}: \Gr_{M_{Q}}\to \Gr_{M_{Q'}}$ whenever $Q'\leq Q$ that satisfy $\varphi_{M_{Q},M_{Q''}} = \varphi_{M_{Q'},M_{Q''}}\varphi_{M_Q,M_{Q'}}$ and $\varphi_{M_Q,M_Q} = \id$.  This allows us to form the limit 
\begin{equation}\label{eq:grMw}
	\tsc{M,w}:=  \invLim{M_Q}{\subd{M}{w}} \tsc{M_Q}
\end{equation}

\noindent  Rather than keeping track of the full face poset of  $\subd{M}{w}$  it is desirable to restrict ourselves to cells of codimension 0 and 1.   
The following construction mimics the definition of adjacency graphs for triangulations of polytopes~\cite[Definition 4.5.10]{DeLoeraRambauSantos}, so we use the same name. 

\begin{definition}\label{def:systemsAndSkeleton}
	Given $w$ in $\dr{M}$, let $\grdiv{M}{w}$ be the \emph{adjacency graph} of $\subd{M}{w}$ defined as follows. The graph $\grdiv{M}{w}$ has a vertex $v_{Q}$ for each $Q$ in $\TC{M,w}$. Two vertices $v_{Q_1}, v_{Q_2}$ are connected by an edge if $Q_1\cap Q_2$ is a facet of both cells. 	
	Similarly, given a cell $F$ of $\subd{M}{w}$, we let $\grdivF{M}{w}{F}$ be the full subgraph of $\grdiv{M}{w}$ generated by those vertices $v_{Q}$ of $\grdiv{M}{w}$ with $F\leq Q$. 
\end{definition}

Our next lemma shows that the graphs defined above are connected. It will play a crucial role in \autoref{pr:FromAllToCodim1Diagram} below.

\begin{lemma}\label{lm:graphsAreConnected} For any $w\in \dr{M}$ and  any cell $F$ of $\subd{M}{w}$, the graphs $\grdiv{M}{w}$ and  $\grdivF{M}{w}{F}$ are connected.
\end{lemma}
\begin{proof} The first claim follows by convexity and is valid for the adjacency graph associated to a pure-dimensional polyhedral subdivision of any polytope. We argue for $\Delta_M$ and $\subd{M}{w}$.  Indeed, given two vertices $v_{Q_1}, v_{Q_2}$ of $\grdiv{M}{w}$, choose two points $x_1,x_2$, with $x_i\in \relint(Q_i)$ so that the segment $[x_1,x_2]$ does not meet any cell whose codimension is $2$ or greater. Since $\Delta_M$ is convex, the $[x_1,x_2]$ lies in $\relint(\Delta_M)$.	
All but finitely many points in $[x_1,x_2]$ lie in the relative interior of top-dimensional cells. We label the encountered cells as we move from $x_1$ towards $x_2$ by $Q_1 =:\! Q'_0, Q'_1, \ldots, Q'_k \!:= Q_2$.  
The collection $\{Q'_i\}_{i}$ yields a path from $v_{Q_1}$ to $v_{Q_2}$ in $\grdiv{M}{w}$.

A similar argument can be used to prove the statement for $\grdivF{M}{w}{F}$. Let $E\subset \RR^n$ be the affine span of $\Delta_M$.  Given $F$ in $\subd{M}{w}$, write $s:=\dim F$ and pick a point $p$ in its relative interior. We let $H$ be the orthogonal complement to the linear  subspace $F-p$ in $E-p$, and $Q$ a $(m-s)$-dimensional cube in $H$ centered at the origin with diameter $0<\varepsilon\ll 1$.
	
We consider the full-dimensional polytope $P':=(Q+p) \cap \Delta_M$ in $H+p$, and the polyhedral subdivision on $P'$ induced by $\subd{M}{w}$.
	Each cell in this subdivision equals $Q'\cap \Delta_M$ for some $Q'$, and has dimension $(s-\dim \Delta_M +\dim Q)$. 
	By construction, a matroid  polytope $Q'$ 
	yields a vertex or edge of $\grdivF{M}{w}{F}$ if and only if $Q'\in \subd{M}{w}$ and $Q'\cap (Q+p)\neq \emptyset$. Thus, $\grdiv{M}{w}$ agrees with  the adjacency graph  of the subdivision of $P'$. Since the latter is connected by the discussion above, the result follows.
\end{proof}

The adjacency graph $\grdiv{M}{w}$ encodes a subsystem of the inverse system  $\tsc{M,w}$ from~\eqref{eq:grMw} as in \Cref{ex:diagramDualGraph}. Our final result shows that $\tsc{M,w}$ agrees with the inverse system induced by $\grdiv{M}{w}$.

\begin{proposition}\label{pr:FromAllToCodim1Diagram}
	Let $M$ be a $\kk$-realizable $(d,[n])$-matroid and $w \in \TGr_{M}$. Then,
	\begin{equation}\label{eq:InverseLimits}
		\tsc{M,w}  
		\cong \varprojlim_{\Gamma_{M,w}} \tsc{M'}.
	\end{equation}
\end{proposition}

\begin{proof} We write $\tsc{M,w}^{\Gamma}$  for the inverse limit on the right-hand side of~\Cref{eq:InverseLimits}. Given $M'$ labeling a cell of $\grdiv{M}{w}$, we write $h_{M'}^{\Gamma}\colon \tsc{M,w}^{\Gamma} \to \tsc{M'}$ for the associated morphism. 
Since $\grdiv{M}{w}$ determines a subsystem of $\subd{M}{w}$, the universal property of $\tsc{M,w}^{\Gamma}$ guarantees the existence of a morphism  $\psi: \tsc{M,w} \to \tsc{M,w}^{\Gamma}$. Next, we build a morphism $\phi: \tsc{M,w}^{\Gamma} \to \tsc{M,w}$. 
	
	First, we  construct morphisms $g_{F}\colon \tsc{M,w}^{\Gamma} \to \tsc{F}$ for each cell $\Delta_{F}$ of $\subd{M}{w}$, satisfying $g_{M''} = \varphi_{M',M''} \circ g_{M'}$ for each pair of cells in $\subd{M}{w}$ with $M'\leq M''$. The morphism $\phi$ will be unique determined once we establish the compatibility of all $g_F$'s with the subdivision $\subd{M}{w}$.
	
	Let $\mathcal{V}_F$ be the set of vertices of the graph $\grdivF{M}{w}{F}$. Set  $g_{F}:=\varphi_{M',F} \circ h_{M'}^{\Gamma}$ where $v_{M'}\in \mathcal{V}_F$. We must show this morphism is independent of our choice of $M'$. Suppose $v_{M''}\in V_F$ as well. 
	Since $\grdivF{M}{w}{F}$ is connected by~\autoref{lm:graphsAreConnected}, we can find a collection of vertices $v_{Q_{M'}}=:\!v_{Q_0}, v_{Q_1}, \ldots, v_{Q_{k}}\!:=v_{Q_{M''}}$ where $(v_{Q_i},v_{Q_{i+1}})$ is an edge of $\grdivF{M}{w}{F}$ for each $i=0,\ldots, k-1$.
	We write $M_i:=Q_{M_i}$, $M_{i(i+1)}:=M_{Q_i\cap Q_{i+1}}$, and note that $F\leq Q_i\cap Q_{i+1}$ for each $i$.
	The definition of inverse limit yields $k$ diagrams
	\begin{equation}\label{eq:facetsForPartial}
		\begin{aligned}\xymatrix{
				& \tsc{M_i}  \ar@/_/[dr]^-{\varphi_{M_i,M_{i(i+1)}}}  \ar@/^1pc/[drrr]^-{\varphi_{M_i,F}}&  & & \\
				\tsc{M,w}^{\Gamma} \ar[ur]^-{h_{M_i}^{\Gamma}} \ar[dr]_-{h_{M_{i+1}}^{\Gamma}} \ar[rr]^-{h_{M_{i(i+1)}}^{\Gamma}} & & \tsc{M_{i(i+1)}} \ar[rr]^{\varphi_{M_{i(i+1)},F}} & & \tsc{F}, \\
				& \tsc{M_{i+1}} \ar@/^/[ur]_-{\varphi_{M_{i+1},M_{i(i+1)}}} \ar@/_1pc/[urrr]_-{\varphi_{M_{i+1},F}} & & &
			}
		\end{aligned}
	\end{equation}
	where all four triangles commute. 
	It follows that  $\varphi_{M',F} \circ h_{M'}^{\Gamma} = \varphi_{M'',F} \circ h_{M''}^{\Gamma}$, so $g_F$ is well-defined.
	
	Finally, the identity  $g_{F'} = \varphi_{F,F'}\circ g_{F}$ for each pair  $F'\leq F$ in $\subd{M}{w}$ follows from a similar commutative diagram argument after choosing a  vertex $M'$ in $\grdivF{M}{w}{F}$. These two properties determine $\phi$.
	The universal property of both schemes $\tsc{M,w}$ and $\tsc{M,w}^{\Gamma}$ ensures that $\phi = \psi^{-1}$, as desired.
\end{proof}

\bibliographystyle{abbrv}
\bibliography{bibliographie}
\label{sec:biblio}
\bigskip

\end{document}